\documentclass[final,leqno,letterpaper]{etna}

\usepackage{amssymb}
\usepackage{amsmath}
\usepackage{mathrsfs}
\usepackage{enumerate,array}
\usepackage{graphicx,epstopdf}
\usepackage{color}
\usepackage{stmaryrd} 
\usepackage{bm} 
\usepackage{caption,subcaption}
\usepackage{algorithm}

 \usepackage{lineno}

\newcommand{\R}{\mathbb{R}}

\newcommand{\Fc}{\mathcal{F}}
\newcommand{\Pc}{\mathcal{P}}

\newcommand{\Qc}{\mathcal{Q}}
\newcommand{\Tc}{\mathcal{T}}

\newcommand{\pmat}[1]{\begin{bmatrix}#1\end{bmatrix}}

\newcommand{\norm}[1]{\left\| #1 \right\|}

\newcommand{\abs}[1]{\left| #1 \right|}
\newcommand{\ip}[1]{\left< #1 \right>}

\newcommand{\avg}[1]{\left\{ #1 \right\}}
\newcommand{\jump}[1]{\ensuremath{[\![#1]\!]} }
\newcommand{\jumpF}[1]{\left[ #1 \right]}

\newcommand{\normgam}[1]{\norm{#1}_{L^2(R)}}

\newcommand{\eps}{{{\bm \varepsilon}}}

\newcommand{\partials}[2]{\frac{\partial #1}{\partial #2}}

\newcommand{\bnhat}{{\bf n}}

\newcommand{\bj}{{j}}
\newcommand{\bX}{{\bf X}}
\newcommand{\bY}{{\bf Y}}
\newcommand{\bx}{{\bf x}}

\newcommand{\bu}{\bm{u}}

\newcommand{\bF}{{\bf F}}

\newcommand{\bv}{ \bm{v}}
\newcommand{\bV}{\bm{V}}
\newcommand{\bT}{{\bf T}}
\newcommand{\bjj}{{j}}

\newcommand{\hs}{{\tilde h}}

\newcounter{problem}

\setbibdata{1}{xx}{46}{2017} 

\hypersetup{%
    pdftitle={An unconditionally stable semi-implicit CutFEM for an interaction problem between an elastic membrane and an incompressible fluid},
    pdfauthor={Kyle Dunn, Roger Lui, Marcus Sarkis},
    pdfkeywords={immersed boundary method, finite element method, numerical stability, CutFEM, unfitted methods
    }}

\title{An unconditionally stable semi-implicit CutFEM
	for an interaction problem between
	an elastic membrane and an incompressible fluid\thanks{%
Received... Accepted... Published online on... Recommended by....
This research was supported by the National Science Foundation Division of Mathematical Sciences Grant No.\ 1522663 Higher-Order Methods for Interface Problems with Non-Aligned Meshes.
}}

\author{Kyle Dunn\footnotemark[2] \footnotemark[3] \and Roger Lui \footnotemark[3]
        \and Marcus Sarkis\footnotemark[5]}

\shorttitle{AN UNCONDITIONALLY STABLE SEMI-IMPLICIT CUTFEM} 
\shortauthor{K.~DUNN, R.~LUI, AND M.~SARKIS}

\begin{document}

\maketitle

\renewcommand{\thefootnote}{\fnsymbol{footnote}}

\footnotetext[2]{Now at Cold Regions Research and Engineering Laboratory - U.S. Army ERDC, Hanover, NH 03755, {\tt (Kyle.G.Dunn@usace.army.mil)}}
\footnotetext[3]{Mathematical Sciences Department, Worcester Polytechnic Institute,
	Worcester, MA 01609, ({\tt \{rlui, msarkis\}@wpi.edu)}}

\begin{abstract}
In this paper we introduce a finite element method for the Stokes equations with a massless immersed membrane.
This membrane applies normal and tangential forces affecting the velocity  and pressure of the fluid.
Additionally, the points representing this membrane move with the local fluid velocity.
We design and implement a high-accuracy cut finite element method (CutFEM) which enables the use of a structured mesh that is not aligned with the immersed membrane and then
we formulate a time discretization that yields an unconditionally energy stable scheme.
We prove that the stability is not restricted by the parameter choices that constrained previous finite element immersed boundary methods and illustrate the theoretical results with numerical simulations.
\end{abstract}



\section{Introduction}

Fluid dynamics problems with immersed boundaries have arisen in many real world scenarios such as cardiac blood flow \cite{Peskin77, Peskin89} and cell mechanics \cite{Agresar98}.
Two prevalent ideas are the \emph{immersed boundary method} introduced by Peskin \cite{Peskin02} and the \emph{immersed interface method} by LeVeque and Li \cite{Leveque94,Leveque97}.
These are both finite difference methods developed for very involved problems.
We note that the immersed interface method was also extended to
the finite element method by imposing flux conservation and
continuity of the solution strongly at certain points of $\Gamma$;
see \cite{Adjerid14,BurmanGuzmanSarkis16,Gong07,Gong10,GuzmanSarkis16,GuzmanSarkis17,He11,He12,Li03,Lin15}.
In the immersed boundary method, the interface applies a local force when computing the fluid velocity and pressure globally at each time step.
The right-hand side function is defined only on the interface and contains a Dirac delta function whose main purpose is to pass information between Eulerian and Lagrangian coordinates.
Peskin's use of a finite difference method requires smoothing of the effects of the force applied by the membrane. Boffi and Gastaldi extended these ideas to the finite element method in \cite{Boffi03}.  In their work, a variational formulation in weak form
is introduced and the action of the forcing function, due to bending and stretching, is now written as an integral over the immersed membrane.
One can also show that when the problem is written in the strong form,
the force applied by the membrane to the fluid is equal to the  jump in the normal stress \cite{Lai01} of the fluid
across the membrane. The conditional energy stability of the method proposed in \cite{Boffi03} was proved later in \cite{Boffi06}.

The framework of our finite element method begins with Nitsche's formulation \cite{Nitsche71} in order to weakly impose Dirichlet boundary conditions on fitted meshes. In
\cite{Hansbo05}, the  Nitsche's formulation was extended to the case
where the domain boundary does not align with the underlying finite element
mesh. In our work, we employ one particular fictitious domain finite element method known as CutFEM \cite{Burman10Ghost,Burman14,Burman10,Burman12} which allows us to divide the global domain into two non-overlapping subdomains. 
This technique not only separates the stress on each side, but also allows us to weakly impose a condition on the jump of the normal derivative in lieu of the prescribed forcing function in the earlier work.
Our numerical experiments show that optimal spatial convergence can be obtained using CutFEM when the interface is described by a static, smooth parameterization.

CutFEM was implemented for Stokes with an immersed boundary and a P1-iso-P2 element in \cite{Hansbo14} where Hansbo et al.\ use a known a priori level set method to track the interface.
In that article it is noted that the optimality of their approach is independent of the interface representation, which moves with a prescribed velocity.
Some additional work has been done on problems with a known interface velocity \cite{Hansbo16}.
Our approach focuses on the movement of the interface not known a priori, that is, we let the interface move with the velocity of the fluid which is not known prior to solving the system at each time step.
Here, we implement a  Q2-P1 time-depedent Stokes element
	\cite{MR3544656, MR3807820, MR1950614, MR2092744, MR2421892} and track the immersed boundary by updating the position of a fixed number of points sampled from the initial curve.
We show that our method using techniques similar to those presented in  \cite{MR1826211,  MR2799400, Newren07, Tu92} on the finite difference immersed boundary method is energy stable. We note that the method presented in this paper
can be extended to two-phase flows.
Work on static interface problems with unfitted meshes for two-phase flows has been done by many groups, see e.g. \cite{Becker09, Burman14,  Chessa03}.

This paper is organized as follows.
Section \ref{Sec: Model formulation} will derive the model and introduce the strong formulation of the spatially continuous problem.
Then in Section \ref{Sec: Discrete time approximation} we look into the time discretization of the problem and prove stability results, building to the fully discretized problem.
In Section \ref{Sec: Discrete-time finite element formulation} we introduce the necessary notation and spaces of functions en route to defining our finite element methods.
We proceed to prove energy stability of the proposed finite element problem, which is unconditional for our semi-implicit method and yields a CFL-like condition for the explicit method.
The results of some numerical tests are shown and discussed in Section \ref{Sec: Numerical results} and we draw conclusions in Section \ref{Sec: Conclusions}.

\section{Model formulation}
\label{Sec: Model formulation}

Consider a domain $\Omega$ in $\R^d$, $d=2,3$, that can be any Lipschitz domain.
For simplicity, we will define $\Omega := (0,1)^2$.
The following equations model an incompressible elastic material inside $\Omega$ using Stokes equations.
The stress tensor is defined by $\bT := -\mu\eps(\bu) + {\bf I} p$ where $\eps(\bu) = \frac{1}{2}( \nabla \bu + (\nabla \bu)^T)$ and $p$ is the pressure. 
To reduce notational clutter, define $\mu$ to be twice the traditional dynamic viscosity.
Inside $\Omega$ there will be a closed curve $\Gamma$ representing a massless, elastic interface between two non-overlapping subdomains $\Omega_1$ and $\Omega_2$.
Throughout this paper, we let $\Omega_1$ denote the region exterior to the curve $\Gamma$ such that $\partial \Omega_1 = \partial\Omega \cup \Gamma$ and $\Omega_2$ denote the interior region encapsulated by $\Gamma$.
 Throughout this work   
$\mu$ is assumed to be constant. We note that
our results, with minor modifications,  hold also for the case where
$\mu$ jumps across $\Gamma$ and varies mildly inside each $\Omega_i$.

The description of the interface $\Gamma$ and the model
for the jump of the stress across $\Gamma$ are based
on the immersed boundary method; see Peskin \cite{Peskin02} and Boffi et al.\ 
\cite{Boffi06}. As we will see in Remark \ref{parametrization}, it is
advantageous to describe $\Gamma$, and therefore the jump of the stress, at time $t$ in parametric form  $\Gamma(s,t)$ for $s\in [0,L]$ and fixed $L$
independent of time.
In general, $s$ is \emph{not} an arc-length
parameterizaton of $\Gamma$ at any time $t$. We use $\bX(s,t)$ to
denote the Cartesian coordinates of $\Gamma(s,t)$ corresponding to a point
$s$ for any given time $t$.
Since this is a closed curve, $\bX(0,t) = \bX(L,t)$ for all time.
To construct $\Gamma(s,t)$, we first define  $\Gamma(s,0)$
given by parametrization $s\in [0,L]$. The fact that $\Gamma$ is not necessarily
parameterized by arc-length allows us to define an initial elastic 
membrane
not only with bending, but also with stretching; that is, 
${\abs{\partial \bX/\partial s}}$ is not necessarily equal to one.
For time $t>0$, we let $\bX(s,t)$ be the material point on the elastic
membrane that moves from an initial position $\bX(s,0)$ and also
we assume that the movement of a point $\bX(s,t)$ on the interface
is given by the fluid velocity at that point. Hence,
we impose continuity of velocity
$$\jump{\bu}=0$$
on the interface. For a quantity $\phi$
defined over $\Omega$ we denote  $\phi_1 = \phi|_{\Omega_1}$ and $\phi_2 = \phi |_{\Omega_2}$. Then $\jump{\phi} = (\phi_1 - \phi_2) |_{\Gamma}$ denotes the jump of $\phi$ across $\Gamma$ at a given point. We also impose a
no-slip condition on the interface, that is,
$$\partials{\bX(s,t)}{t} = \bu(\bX(s,t),t).$$

The unit tangent vector, chosen to be in the direction of the parameterization, is defined in terms of $s$ by
$${\bf\tau} = \frac{1}{\abs{\partial \bX/\partial s}} \partials{\bX}{s}.$$
The boundary tension $T(s,t)$ of the elastic membrane is modeled using a generalized Hooke's law, where
$$T(s,t) = \sigma\left(\abs{\partial \bX/\partial s};s,t\right)$$
and the function 
$\sigma$ is defined below.  By computing the elastic force on an arbitrary segment between two points $a$ and $b$, we find that
$$(T{\bf \tau})(b,t) - (T{\bf \tau})(a,t) = \int_a^b \partials{}{s} (T{\bf \tau})(s,t)\, ds.$$
Since this equality holds for any choice of $a$ and $b$, we know the force on $\Gamma$ is defined in terms of $s$ by
\begin{equation}
\bF = \partials{}{s} (T{\bf \tau}).
\label{F}
\end{equation}

A slight modification of the proof of Theorem 1 in \cite{Lai01} shows that for a force $\bF$ defined in terms of $s\in [0,L]$,
	\begin{equation*}
	\begin{array}{l}
	\displaystyle\jump{p\bnhat} = -\frac{(\bF\cdot\bnhat)\bnhat}{\abs{\partial \bX/\partial s}}\\
	\displaystyle\mu\jump{\eps(\bu) \bnhat} =  \frac{\bF - (\bF\cdot\bnhat)\bnhat}{\abs{\partial \bX/\partial s}}.
	\end{array}
	\end{equation*}
It follows from \eqref{F} that if we choose $\sigma(\abs{\partial \bX/\partial s};s,t)$ to be proportional to $\abs{\partial \bX/\partial s}$, i.e., $\sigma(\abs{\partial \bX/\partial s};s,t) = \kappa \abs{\partial \bX/\partial s}$, then the jump condition is defined by
	\begin{equation}
	\jump{(\mu\eps(\bu)-p)\bnhat} = \frac{\bF(s,t)}{\abs{\partials{\bX}{s}(s,t)}} =  \frac{\kappa}{\abs{\partials{\bX}{s}(s,t)}} \partials{^2\bX}{s^2}(s,t).
	\label{stress jump}
	\end{equation}
Physically, \eqref{stress jump} means that the elastic interface will apply a force as it is stretched or bent at a given point.
Here, the jump condition is defined in terms of the respective quantities restricted to $\Gamma(t)$.
For example,
\begin{align*}
\jump{(\mu\eps(\bu)-p)\bnhat} &= \big((\mu\eps(\bu_1(\bX(s,t),t))-p_1(\bX(s,t),t))\\
& \qquad \quad - (\mu\eps(\bu_2(\bX(s,t),t))-p_2(\bX(s,t),t))\big)\bnhat(\bX(s,t),t)
\end{align*}
for any $s\in [0,L]$.
To ease the notation, we denote $\bnhat = \bnhat_1$, i.e., the unit normal pointing outward from the exterior $\Omega_1$.

We further impose a homogeneous Dirichlet condition on $\partial\Omega$.
Combining Stokes equations with the continuity of velocity and \eqref{stress jump}, the strong form of the equations to be solved is given in Problem \ref{Ch2: Strong form}.

\begin{algorithm}
	\refstepcounter{problem}
	{\bf Problem \theproblem: Strong formulation}\\
	
	Find $\bX(s,t)$, $\bu_i(\bx,t)$, and $p_i(\bx,t)$ for $i=1,2$ such that for all $t\in (0,T)$
	\begin{subequations}
		\begin{alignat}{2}
		\partials{\bu_i}{t}-\mu\nabla\cdot\eps(\bu_i) + \nabla p_i &= 0 \qquad &&\text{ in } \Omega_i(t), \, i = 1, 2 ,\label{ContTime: elasticity}\\
		\nabla\cdot \bu_i &= 0 \qquad &&\text{ in } \Omega_i(t), \, i = 1,2, \label{ContTime: incompressibility}\\
		\jump{(\mu\eps(\bu)- p)\bnhat} &=  \frac{\kappa}{\abs{\partials{\bX}{s}}} \partials{^2\bX(s,t)}{s^2} \qquad &&\text{ for } s\in [0,L],\label{ContTime: stress jump}\\
		\jump{\bu} &= 0 \qquad && \text{ for } s\in[0,L], \label{ContTime: velocity jump}\\
		\bu_1 &= 0 \qquad &&\text{ on } \partial \Omega,\label{ContTime: Dir BC} \\
		\partials{\bX(s,t)}{t} &=   \bu(\bX(s,t),t)   \qquad && \text{ for } s\in[0,L]. \label{ContTime: movement condition}
		\end{alignat}
	\end{subequations}
	\label{Ch2: Strong form}
\end{algorithm}

\begin{remark}
	Note the time dependence of each subdomain and the location of the interface.
	When deriving the weak formulation, our spaces of test functions depend on time as well.
\end{remark}

\begin{remark} \label{parametrization} 
	In Nitsche's formulation of the interface problem, we must substitute \eqref{ContTime: stress jump} into an integral over $\Gamma(t)$.	
	We note that in its original form, we are integrating with respect to a time-dependent arc length parameterization of the interface.
	Since \eqref{ContTime: stress jump} is defined in terms of $s$, we will transform this integral over $\Gamma(t)$ to an integral over $[0,L]$.
	We have the following equalities:
	\begin{align*}
	\big(\jump{(\mu\eps(\bu)- p)\bnhat},\{\bv\}\big)_{\Gamma(t)} &= \int_{\Omega}\jump{(\mu\eps(\bu)- p)\bnhat}\cdot \left\{\bv\left( \bx \right)\right\}\delta(\bx - \bX(s,t)) \, d\bx\\
	&= \int_{0}^L \jump{(\mu\eps(\bu)- p)\bnhat}\cdot \left\{\bv\left( \bX(s,t) \right)\right\} \abs{\partials{\bX}{s}}\, ds\\
	&= -\int_{0}^L \kappa \partials{\bX}{s}(s,t)\cdot \partials{}{s}\left\{\bv\left( \bX(s,t) \right)\right\} \, ds,
	\end{align*}
	where we have denoted the average of a function $\phi$  by $\avg{\phi} = \frac{1}{2}(\phi|_{\Omega_1} + \phi|_{\Omega_2})$.
	
	The weak formulation can be obtained by the usual integration by parts on \eqref{ContTime: elasticity}-\eqref{ContTime: incompressibility} after multiplication by a test function.
	To symmetrize the problem for increased accuracy and computational efficiency, we add consistent terms to the weak formulation, seen in \cite{Burman09,Hansbo05,Hansbo14} .
	A nonsymmetric interior penalty method may also be used, see e.g. \cite{Burman09,Hansbo14}.
\end{remark}


\section{Discrete-Time Approximation}
\label{Sec: Discrete time approximation}

Given $\Delta t$, we consider equally-spaced time steps $t_n = n\Delta t$, for $0\leq n \leq N_t$.
For each $n$, let $\Gamma^n = \Gamma(t_n)$ be the interface separating the two subdomains $\Omega_i^n = \Omega_i(t_n)$.
We also let $\bu^n = \bu(\bx,t_n)$ and $p^n = p(x,t_n)$ to simplify notation.
Below in the temporally discrete variations of \eqref{ContTime: elasticity}-\eqref{ContTime: movement condition}, we use a backward-difference approximation for $\partial_t \bu_i$.  
In other words, the derivative with respect to time at $t_{n+1}$ is approximated by $\partial_t \bu_i^{n+1} = (\bu_i^{n+1}-\bu_i^n)/\Delta t$.
 Because $\bu_i^n$ is defined on $\Omega_i^{n-1}$ but must be integrated over $\Omega_{i}^n$, we define
$$\tilde \bu^n := \left\{\begin{array}{ll}
\bu_1^n(\bx),& \quad \text{ for } \bx\in \Omega_1^{n-1},\\
bu_2^n(\bx),& \quad \text{ for } \bx\in \Omega_2^{n-1}.
\end{array}
\right.$$

Recall that each integral over $\Gamma^n$ will be expressed in terms of $s$.
We also write $\bX^n(s)= \bX(s,t_n)$ to simplify the notation.
In the spatially continuous case, we simplify the inner products involving the jump condition on the interface as follows:
\begin{itemize}

	\item Explicit method:
	\begin{equation}
	\big(\jump{(\mu\eps(\bu)-p)\bnhat},\{\bv\}\big)_{\Gamma^{n}} = -\kappa\int_0^L \partials{\bX^{n}(s)}{s} \partials{}{s}\avg{\bv(\bX^{n}(s))}\, ds.
	\label{Explicit Integral}
	\end{equation}
	
	\item Semi-implicit method:
	\begin{align}
	\big(&\jump{(\mu\eps(\bu)-p)\bnhat},\{\bv\}\big)_{\Gamma^{n}}
	= -\kappa\int_0^L \partials{\bX^{n+1}}{s} \partials{}{s}\avg{\bv(\bX^{n}(s))}\, ds\nonumber\\
	&\,\quad = -\kappa\int_0^L \left(\partials{\bX^n}{s} + \Delta t \partials{}{s}\avg{\bu^{n+1}(\bX^n(s))}\right) \partials{}{s}\avg{\bv(\bX^{n}(s))}\, ds.
	\label{Semi-implicit integral}
	\end{align}
\end{itemize}
See that the difference between \eqref{Explicit Integral} and \eqref{Semi-implicit integral} is the extrapolation used in the semi-implicit method, where we solve for $\bX^{n+1}$ in \eqref{ContTime: movement condition}.
Note that $\bu^{n+1}(\bX^n) = \avg{\bu^{n+1}(\bX^n)}$ in the spatially continuous problem and the average is included for comparison to the discrete case.
The expression for the forcing function \eqref{Semi-implicit integral} incorporates the unknown velocity of the interface at the current time step.

We formulate the continuous-space, discrete-time Problem \ref{prob: continuous weak form} letting $\bY = \bX^{n+1}$ for the semi-implicit method and $\bY = \bX^n$ for the explicit method. 
For completeness, we note that the implicit method sets $\bY = \bX^{n+1}$ and integrates over $\Omega^{n+1}$ and $\Gamma^{n+1}$ in Problem \ref{prob: continuous weak form} instead of $\Omega^{n}$ and $\Gamma^{n}$, respectively.
Additional challenges arise with the implicit method because $\Gamma^{n+1}$ and $\Omega^{n+1}$ are not known prior to integration, so we omit this discretization.
We included the terms involving $\jump{\bu^{n+1}}$ in Problem \ref{prob: continuous weak form}
to compare to the one further discretized in space, Problem \ref{Problem: FEM}, although $\jump{\bu^{n+1}} = 0$ in the current continuous setting.
For the same reason, we include $\avg{\partials{}{s}\bu^{n+1}(\bX^n(s))} =
\partials{}{s}\bu^{n+1}(\bX^n(s))$. 

\begin{algorithm}[h]
	
	\refstepcounter{problem}
	{\bf Problem \theproblem: Discrete-time weak formulation}\\
	\noindent Given $(\bu_1^0,\bu_2^0)\in [H^1(\Omega_1^0)]^2 \times [H^1(\Omega_2^0)]^2$ where $\bu_1^0|_{\partial\Omega} = 0$, $(p_1^0,p_2^0) \in (L^2(\Omega_1^0)\times L^2(\Omega_2^0))/\R$, and $\bX^0: [0,L]\to \Omega$, find for all $1
	\leq n \leq N_t-1$ solutions $(\bu_1^{n+1},\bu_2^{n+1})\in [H^1(\Omega_1^{n})]^2 \times [H^1(\Omega_2^{n})]^2$, $(p_1^{n+1},p_2^{n+1}) \in \big(L^2(\Omega_1^{n})\times L^2(\Omega_2^{n})\big)/\mathbb{R}$, and $\bX^{n+1}: [0,L]\to \Omega$ such that $\bu_1^{n+1} = 0$ on $\partial\Omega$, $\jump{\bu^{n+1}} = 0$ on $\Gamma^n$, and 
	\begin{subequations}
		\begin{alignat}{1}
		\begin{split}  \label{Dynamic: discrete time formulation}
		\sum_{i=1}^2 \frac{1}{\Delta t}(&\bu_i^{n+1}-\tilde\bu^n,\bv_i)_{\Omega_i^{n}} + \mu(\eps(\bu_i^{n+1}),\eps(\bv_i))_{\Omega_{i}^{n}} - (p_i^{n+1}, \nabla \cdot \bv_i)_{\Omega_{i}^{n}}
		\\
		&  - \big(\{(\mu\eps(\bu^{n+1})-p^{n+1})\bnhat\},\jump{\bv}\big)_{\Gamma^{n}} - \big(\jump{\bu^{n+1}},\{\mu\eps(\bv)\bnhat\}\big)_{\Gamma^{n}}\\ 
		&-	\big((\mu\eps(\bu_1^{n+1})-p_1^{n+1})\bnhat,\bv_1\big)_{\partial\Omega} -\big(\bu_1^{n+1},\mu\eps(\bv_1)\bnhat\big)_{\partial\Omega}
		\\
		&= \int_{0}^L \kappa \partials{^2\bY}{s^2}(s) \left\{\bv\left( \bX^{n}(s) \right)\right\} \, ds  \end{split}
		\\
		\sum_{i=1}^2 -(\nabla& \cdot \bu_i^{n+1}, q_i)_{\Omega_{i}^{n}} + \big(\jump{\bu^{n+1}},\{q\bnhat\}\big)_{\Gamma^n} + (\bu_1^{n+1},q_1\bnhat)_{\partial\Omega} = 0
		\label{Dynamic: discrete time div equation}
		\end{alignat}
		for all $(\bv_1,\bv_2)\in [H^1(\Omega_1^n)]^2 \times [H^1(\Omega_2^n)]^2$ and $(q_1,q_2) \in L^2(\Omega_1^n)\times L^2(\Omega_2^n)$, and
		\begin{alignat}{2}
		\frac{\bX^{n+1}(s) - \bX^n(s)}{\Delta t} &= \{\bu^{n+1}(\bX^{n}(s))\} &&\qquad \text{ for } s \in [0,L].
		\label{Dynamic: discrete time boundary movement}
		\end{alignat}
	\end{subequations}
	\label{prob: continuous weak form}
\end{algorithm}

\subsection{Energy Estimates}

The proposed semi-implicit method combines the analytical simplicity and stability of the implicit method in \cite{Boffi06} with the computational convenience of the explicit method.
For a quantity $\phi(s)$ defined on $\Gamma$, we define the norm over the reference configuration $R = [0,L]$ to be

\begin{equation}
\label{R norm}
\normgam{\phi}^2 := \int_{0}^{L} \big(\phi(s)\big)^2 \, ds.
\end{equation}
	If we define total energy to be the sum of the kinetic and elastic energies 
	\begin{equation}
	\label{energy definition}
	E^n := \frac{1}{2}\norm{\bu^{n}}^2_{L^2(\Omega)} +
	\frac{1}{2}\kappa\normgam{\partials{\bX^{n}}{s}}^2,
	\end{equation}
then the following lemma shows that the energy of the system computed using $\bY = \bX^{n+1}$ is monotonically decreasing.

\begin{lemma}\label{Dynamic: semi-implicit theorem}
	Let $\bu^{n+1}$, $p^{n+1}$, and $\bX^{n+1}$ be solutions to \eqref{Dynamic: discrete time formulation}-\eqref{Dynamic: discrete time boundary movement} at time $t^{n+1}$ with $\bY = \bX^{n+1}$.  Then the following equality holds:
		\begin{align}
		\begin{split} \label{Dynamic: semi-implicit inequality}
		E^{n+1} &= E^n - \frac{1}{2}\norm{\bu^{n+1} - \bu^n}^2_{L^2(\Omega)} - \Delta t \sum_{i=1}^2 \mu \norm{\eps(\bu_i^{n+1})}^2_{L^2(\Omega)}\\
		&\quad  - \frac{1}{2}\kappa \normgam{\partials{\bX^{n+1}}{s} - \partials{\bX^{n}}{s}}^2.
		\end{split}
		\end{align}
\end{lemma}
\begin{proof}
	Begin by letting $\bv = \bu^{n+1}$ and $q = p^{n+1}$ in \eqref{Dynamic: discrete time formulation}-\eqref{Dynamic: discrete time div equation} and subtract \eqref{Dynamic: discrete time div equation} from \eqref{Dynamic: discrete time formulation},
	where $\bu^{n+1}$ stands for $\bu^{n+1} = \bu^{n+1}_i$ on $\Omega_i^n$ for $i=1,2$. We note that for each time step we have  $\jump{\bu^{n+1}} = 0$ on $\Gamma^n$ and $\bu_1^{n+1}=0$ on $\partial\Omega$.
	Thus, these boundary terms disappear in \eqref{Dynamic: discrete time formulation}-\eqref{Dynamic: discrete time div equation}.
	Using the symmetry of the bilinear form we are able to simplify the difference of \eqref{Dynamic: discrete time formulation} and \eqref{Dynamic: discrete time div equation} to
	\begin{alignat}{1}
	\begin{split} \label{Continuous space: v equals u}
	&\frac{1}{\Delta t}(\bu^{n+1}-\bu^n ,\bu^{n+1})_{\Omega} + \mu(\eps(\bu^{n+1}),\eps(\bu^{n+1}))_{\Omega}\\		
	&\qquad +\kappa\int_{0}^L \partials{\bX^{n+1}}{s}(s) \partials{}{s}\left\{\bu^{n+1}\left( \bX^{n}(s) \right)\right\} \, ds = 0.
	\end{split}
	\end{alignat}
	First, we can rewrite
	\begin{align*}
	\sum_{i=1}^2 \mu(\eps(\bu_i^{n+1}),\eps(\bu_i^{n+1}))_{\Omega_{i}^{n}} &= \sum_{i=1}^2 \mu\norm{\eps(\bu_i^{n+1})}^2_{L^2(\Omega_{i}^{n})}\\
	&= \mu \norm{\eps(\bu^{n+1})}^2_{L^2(\Omega)}.
	\end{align*}
	Now simplifying the forcing term in \eqref{Continuous space: v equals u}, we have
	\begin{align*}
	\kappa&\int_{0}^L \partials{\bX^{n+1}}{s}(s) \partials{}{s}\left\{\bu^{n+1}\left( \bX^{n}(s) \right)\right\} \, ds\\
	&= \frac{\kappa}{\Delta t}\int_0^L \partials{\bX^{n+1}}{s}\left(\partials{\bX^{n+1}}{s}-\partials{\bX^{n}}{s}\right)\, ds\\
	&= \frac{\kappa}{2\Delta t}\int_0^L \left(\partials{\bX^{n+1}}{s} + \partials{\bX^{n+1}}{s} - \partials{\bX^{n}}{s} + \partials{\bX^{n}}{s}\right)\left(\partials{\bX^{n+1}}{s}-\partials{\bX^{n}}{s}\right)\, ds\\
	&= \frac{\kappa}{2\Delta t}\int_0^L \left[ \left(\partials{\bX^{n+1}}{s}\right)^2 + \left(\partials{\bX^{n+1}}{s} - \partials{\bX^{n}}{s}\right)^2 - \left(\partials{\bX^{n}}{s}\right)^2 \, \right] ds\\
	&=\frac{\kappa}{2\Delta t}\left( \normgam{\partials{\bX^{n+1}}{s}}^2 +  \normgam{\partials{\bX^{n+1}}{s} - \partials{\bX^{n}}{s}}^2 -  \normgam{\partials{\bX^{n}}{s}}^2\right). 
	\end{align*}
	Using a similar manipulation for the first term on the left-hand side of \eqref{Continuous space: v equals u}, we obtain by a simple calculation that
	\begin{align*}
	(\bu^{n+1}-\bu^n ,\bu^{n+1})_{\Omega} = \frac{1}{2}\left( \norm{\bu^{n+1}}^2_{L^2(\Omega)}+  \norm{\bu^{n+1}-\bu^{n}}^2_{L^2(\Omega)} -  \norm{\bu^{n}}^2_{L^2(\Omega)}\right).
	\end{align*}
	Applying the above simplifications to each term in \eqref{Continuous space: v equals u} and multiplying by $2\Delta t$ we have \eqref{Dynamic: semi-implicit inequality}.
\end{proof}


We now turn to the explicit method, whose solution must satisfy equations \eqref{Dynamic: discrete time formulation}-\eqref{Dynamic: discrete time boundary movement} with $\bY = \bX^n$.
The velocity $\bu^{n+1}$ and pressure $p^{n+1}$ are computed by explicitly using the interface location $\Gamma^n$ and subdomains $\Omega_i^{n}$ determined in the previous time step.
The energy estimate for the explicit method is similar to that of the semi-implicit method, but lacks the stabilizing contribution of the extrapolation used to compute the force of the membrane in \eqref{Semi-implicit integral}.
We have the following lemma.

\begin{lemma} \label{Lemma: explicit}
	Let $\bu^{n+1}$, $p^{n+1}$, and $\bX^{n+1}$ be solutions to \eqref{Dynamic: discrete time formulation}-\eqref{Dynamic: discrete time boundary movement} at time $t^{n+1}$ with $\bY = \bX^n$.  Then the following equality holds:
		\begin{align}
		\begin{split} \label{Dynamic: explicit inequality}
		E^{n+1}& =  E^n  - \frac{1}{2} \norm{\bu^{n+1}-\bu^n}^2_{L^2(\Omega)} - \Delta t \sum_{i=1}^2 \mu \norm{\eps(\bu^{n+1})}^2_{L^2(\Omega)}
		\\
		&\qquad  +  \frac{\kappa}{2}(\Delta t)^2\norm{\nabla_\Gamma\bu^{n+1}}_{L^2(\Gamma^n)}^2.
		\end{split}
		\end{align}
\end{lemma}
\begin{proof}
	We begin with the simplification made in the previous proof:
	\begin{align}
	\begin{split} \label{Dynamic: explicit energy sub}
	\frac{1}{\Delta t}&(\bu^{n+1}-\bu^n ,\bu^{n+1})_{\Omega} + \mu(\eps(\bu^{n+1}),\eps(\bu^{n+1}))_{\Omega}\\
	&+ \kappa\int_{0}^L \partials{\bX^{n}}{s}(s) \partials{}{s}\left\{\bu^{n+1}\left( \bX^{n}(s) \right)\right\} = 0. 
	\end{split}
	\end{align}
	The proof for the explicit case is identical to the proof in the semi-implicit case with one important difference in the treatment of the final term in \eqref{Dynamic: explicit energy sub}.
	We have
	\begin{align*}
	\kappa\int_{0}^L &\partials{\bX^{n}}{s}(s) \partials{}{s}\left\{\bu^{n+1}\left( \bX^{n}(s) \right)\right\}\\
	&= \frac{\kappa}{\Delta t}\int_0^L \partials{\bX^{n}}{s}\left(\partials{\bX^{n+1}}{s}-\partials{\bX^{n}}{s}\right)\, ds\\
	&= \frac{\kappa}{2\Delta t}\int_0^L \left(\partials{\bX^{n+1}}{s} - \partials{\bX^{n+1}}{s} + \partials{\bX^{n}}{s} + \partials{\bX^{n}}{s}\right)\left(\partials{\bX^{n+1}}{s}-\partials{\bX^{n}}{s}\right)\, ds\\
	&= \frac{\kappa}{2\Delta t}\int_0^L \left[ \left(\partials{\bX^{n+1}}{s}\right)^2 - \left(\partials{\bX^{n+1}}{s} - \partials{\bX^{n}}{s}\right)^2 - \left(\partials{\bX^{n}}{s}\right)^2 \, \right] ds\\
	&=\frac{\kappa}{2\Delta t} \left(\normgam{\partials{\bX^{n+1}}{s}}^2 - \normgam{\partials{\bX^{n+1}}{s} - \partials{\bX^{n}}{s}}^2 - \normgam{\partials{\bX^{n}}{s}}^2 \right).
	\end{align*}
	The simplification of the last term in \eqref{Dynamic: explicit energy sub} shown above is almost identical to Lemma \ref{Dynamic: semi-implicit theorem}, but the important difference is that the middle term in the final line above is negative.
	Now the energy may not be decreasing.
	
	To make more sense of the norm involving both $\bX^{n+1}$ and $\bX^n$, we can write it in terms of the surface gradient of the velocity on $\Gamma$ as follows:
	\begin{align*}
	\frac{\kappa}{2\Delta t} \normgam{\partials{\bX^{n+1}}{s} - \partials{\bX^{n}}{s}}^2 &= \frac{\kappa}{2\Delta t} \normgam{\Delta t \partials{}{s}\bu^{n+1}(\bX^n)}^2\\
	&=\frac{\kappa\Delta t}{2}\norm{\nabla_\Gamma\bu^{n+1}}_{L^2(\Gamma^n)}^2.
	\end{align*}
	We substitute the final expression into \eqref{Dynamic: explicit energy sub} along with the simplification of the time-derivative term in the proof of Lemma \ref{Dynamic: semi-implicit theorem} to get \eqref{Dynamic: explicit inequality}.
\end{proof}

We note that for the discrete case, a trace theorem and
	an inverse inequality can be used on
	$\norm{\nabla_\Gamma\bu^{n+1}}_{L^2(\Gamma^n)}$ to establish
	conditional stability;
	see \cite{Boffi06}.
	From now on, we focus only on the unconditionally stable
	semi-implicit method since the explicit case can be treated similarly.

\section{Discrete-Space Finite Element Approximation}
\label{Sec: Discrete-time finite element formulation}

The spatial discretization of the problem requires two steps.
First, the interface $\Gamma$ is discretized.
Recall that we create a mapping from the interval $[0,L]$ to $\Gamma(s,t)$ with $\bX(0,t) = \bX(L,t)$ that is not necessarily an arc-length parameterization.
We choose equally-spaced points 
$0=s_0 < s_1 < \cdots < s_m = L$ by letting $\hs = L/m$ and $s_j = j\hs$.
We note that the set of points $\{s_j\}_{j=0}^m$ need not be evenly spaced, but is chosen so for computational convenience.
Then the initial immersed boundary $\Gamma^0$ is approximated by a polygon $\Gamma_\hs^0$ with $m$ vertices, where the $j$th vertex is obtained by evaluating $\bX^0_j = \bX(s_j,0)$.
While $\{s_j\}_{j=0}^m$ may be equally-spaced for computational convenience, ${\text{dist}}(\bX^n(s_j),\bX^n(s_{j+1}))$ may not be uniform.

Second, we discretize the bulk fluid.
The polygonal approximation $\Gamma_\hs^0$ divides $\Omega$ into the two approximate subdomains.
As the discrete interface moves, these subdomains will change and are denoted by $\Omega_{i,\hs}^n$ at time $t_n$.
Let $\Tc_h$ partition $\Omega$ into squares with side length $h$.
Then the subset of $\Tc_h$ that overlaps each $\Omega_{i,\hs}^n$ is denoted by
$$\Tc_{i,h}^n := \{K\in \Tc_h : \text{meas}_2(K\cap \Omega_{i,\hs}^n)>0\},$$
where $\text{meas}_d$ denotes the Lebesgue measure in $d$ dimensions.
These sets of elements are further decomposed into two disjoint sets, $\Tc_{i,h}^{n,I}$ and $\Tc_{h}^{n,\Gamma}$.
We define  the set of elements of $\Tc_{i,h}^n$ strictly interior to $\Omega_{i,\hs}$ by
$$\Tc_{i,h}^{n,I} := \{K\in \Tc_{i,h}^n : K\subset \Omega_{i,\hs}^n\}.$$
Similarly, the set of elements of $\Tc_{i,h}^n$ whose interior is intersected by the interface $\Gamma_\hs^n$ is defined by
$$\Tc_{h}^{n,\Gamma} := \{K\in \Tc_{h} : \text{meas}_1(K\cap \Gamma_\hs^n)>0\}.$$
Thus, for each $i$ and $n$ the relationship $\Tc_{i,h}^n = \Tc_{i,h}^{n,I} \cup \Tc_{h}^{n,\Gamma}$ holds.
Consider the union of all elements in $\Tc_{i,h}^n$; we define the interior of each union to be the extended subdomain $\Omega_{i,h}^{n,e}$.
These subdomains $\Omega_{i,h}^{n,e}$ depend on both $\Gamma_\hs^n$ and $h$ and can be formally defined by
$$\Omega_{i,h}^{n,e} := \text{Int}\left(\bigcup_{K\in \Tc_{i,h}^n} K\right).$$
Many approximate quantities depend on $h$ and $\hs$, although only one is used as a subscript.
For example, $\bu_{i,h}^n$, $p_{i,h}^n$, $\Tc_{h}^{n,\Gamma}$, and others depend on both $h$ and $\hs$.
The set of points $\{\bX_{j}^n\}_{j=0}^m$ depends only on $\hs$ and the polygon will be refined as $\hs$ decreases for fixed $L$.

\subsection{Finite element problem}

For each set of elements $\Tc_{i,h}^n$  we define the finite element spaces
$$\bV_{i,h}^n := \big\{\bv\in [C^0(\Omega_{i,h}^{n,e})]^2 : \bv|_K \in [\Qc^2(K)]^2, \, \forall K \in \Tc_{i,h}^n\big\}$$
and
$$M_{i,h}^n := \{q\in L^2(\Omega_{i,h}^{n,e}) : q|_K\in \Pc^1(K), \, \forall K \in \Tc_{i,h}^n\}.$$ 
Recall $\Pc^1(K)$ is the space of linear functions defined on an element $K$.
A general $q\in M_{i,h}$ is discontinuous across each edge of the elements since a linear function in two variables is defined by its value at three points.
The space $\Qc^2(K)$ consists of biquadratic functions defined on the element $K$.
These functions have nine local degrees of freedom and are continuous across the edges of each element.

Additional ``ghost" penalty terms are included to mitigate the jumps of the flux and pressure across the faces of elements, particularly to minimize spiking at the ghost nodes and spurious oscillations.
To add these to the minimizing functional, we first need to define the sets of edges over which these jumps will be minimized, denoted $\Fc_{i,h}^{n,\Gamma}$.
Informally, we describe each $\Fc_{i,h}^{n,\Gamma}$ as the union of all edges shared by two elements, where at least one of the elements is in $\Tc_{h}^{n,\Gamma}$.
Formally, these sets are defined by
$$\Fc_{i,h}^{n,\Gamma} = \{K\cap K' : K \neq K', \, \text{ and } \, K\in \Tc_{h}^{n,\Gamma}, \, K' \in \Tc_{i,h}^n\}.$$
Figure \ref{fig: example meshes} shows $\Fc_{i,h}^{n,\Gamma}$ for each subdomain.

Let $K$ and $K'$ be adjacent square elements with $F=K\cap K'$ and define $\phi$ on $K$ and $\phi'$ on $K'$.
Below, $[\phi] = \phi |_{F} - \phi' |_{F}$ denotes the jump of a function over the face $F$.
Then the stabilizing ghost penalty terms are defined by
\begin{align*}
\bj_{i,h}(\bu_{i,h},\bv_{i,h}) &= \sum_{\ell=0}^1 \sum_{F\in \Fc_{i,h}^{n,\Gamma}} \int_{F} h^{2\ell+1} \jumpF{\partial_{\bnhat_F}^{(\ell)}(\eps(\bu_{i,h})\bnhat_F)}\cdot \jumpF{\partial_{\bnhat_{F}}^{(\ell)}(\eps(\bv_{i,h})\bnhat_F)},
\\
J_{i,h}(p_{i,h},q_{i,h}) &= \sum_{\ell=0}^1 \sum_{F\in \Fc_{i,h}^{n,\Gamma}} \int_{F}  h^{2\ell+1} \jumpF{\partial_{\bnhat_{F}}^{(\ell)} p_{i,h}} \jumpF{\partial_{\bnhat_{F}}^{(\ell)}q_{i,h}}.
\end{align*}

\begin{figure}[]
	\begin{center}
		\makebox[\textwidth][c]{\includegraphics[width=\textwidth]{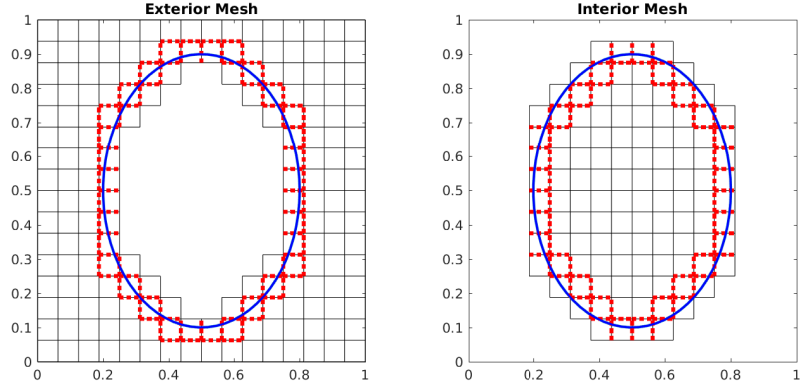}}
		\caption{Plot of example meshes with $\Fc_{i,h}^{n,\Gamma}$ highlighted with a red dotted line for each subdomain.}
		\label{fig: example meshes}
	\end{center}
\end{figure}

Since the interface cuts through elements, we must weakly impose interface conditions across $\Gamma_\hs^n$.
The jump of the stress is incorporated naturally by substitution into the integral resulting from integration by parts.
To impose the weak interface continuity condition and also
	the weak flux continuity condition, we must add mathematically consistent penalty terms
\begin{equation}
\gamma_1 \frac{\mu}{h} \int_{\Gamma_\hs^n} \jump{\bu_h}\cdot\jump{\bv_h}\quad
 \mbox{and}\quad \gamma_2  \frac{h}{\Delta t} \int_{\Gamma_\hs^n} \jump{\bu_h\cdot\bnhat}\jump{\bv_h\cdot\bnhat}
\label{jump of u}
\end{equation}
for some  $\gamma_1>0$ and $\gamma_2>0$. We note that the second
	penalty term above is required to establish the inf-sup condition
	when $h/{\Delta t}$ dominates $\mu/h$.
If $\bu$ is the exact solution, the jump of $\bu$ is equal to zero and \eqref{jump of u} will vanish for the velocity satisfying the system of equations \eqref{ContTime: elasticity}-\eqref{ContTime: movement condition}.
Thus, addition of \eqref{jump of u} will keep the variational formulation consistent with the original problem.
We similarly enforce the weak Dirichlet boundary condition and zero-flux on $\partial\Omega$ by adding the penalty terms
\begin{equation}
\gamma_1 \frac{\mu}{h} \int_{\partial\Omega} (\bu_{1,h}- {\bf 0}) \cdot \bv_{1,h}
\quad \mbox{and} \quad \gamma_2 \frac{h}{\Delta t} \int_{\partial\Omega}
	(\bu_{1,h} - {\bf 0})\cdot\bnhat \,\bv_{1,h}\cdot\bnhat.
\label{jump of u square}
\end{equation}

The parameterization coordinate of the $j$th vertex of $\Gamma_\hs^n$ is denoted $s_j$, where $0\leq j\leq m$, and the corresponding Cartesian coordinate pairs are $\bX_{\hs,j}^n = \bX_\hs(s_j,n\Delta t)$.
Additionally $s_0 = 0$ and $s_m = L$ so that $\bX_{\hs,0}^n = \bX_{\hs,m}^n$ for all $n$.
To ease the notation, we will let $\bX_j^n$ denote the coordinate pair $\bX_{\hs,j}^n$ on the discrete interface.

%

	For both explicit and semi-implicit temporal discretizations, we will find the following simplification of \eqref{Dynamic: discrete time formulation} with $\bY = \bX^n$ useful.
	After integration by parts on $\Gamma_\hs^n$, we simplify the right-hand side of \eqref{Dynamic: discrete time formulation} using the fact that $\partials{\bX^n}{s}$ is constant on each edge of the polygon $\Gamma_\hs^n$.
	If we define
	$$\partials{\bX^n_j}{s} = \frac{\bX^n_{j+1} - \bX^n_{j}}{s_{j+1}-s_j},$$
	then the resulting simplification is written
	\begin{align}
	-\kappa\int_0^L \partials{\bX^n}{s} \partials{}{s}\left\{\bv(\bX_\hs^{n}(s))\right\}\, ds
	&= -\kappa\sum_{j=0}^{m-1} \partials{\bX^n_{j}}{s} \int_{s_j}^{s_{j+1}} \partials{}{s}\left\{\bv(\bX_\hs^{n}(s))\right\}\, ds
	\nonumber\\
	&= -\kappa\sum_{j=0}^{m-1} \partials{\bX^n_j}{s} \left(\left\{\bv(\bX^{n}_{j+1})\right\} - \left\{\bv(\bX^{n}_{j})\right\}\right)
	\nonumber\\
	&= \kappa\sum_{j=0}^{m-1} \left(\partials{\bX^n_{j+1}}{s} - \partials{\bX^n_j}{s} \right) \left\{\bv(\bX^{n}_{j+1})\right\}.
	\label{force summation}
	\end{align}

To simplify notation in Problem \ref{Problem: FEM} we drop the ``$h$" or ``$\hs$" subscript from the discrete approximation of 
the subdomains $\Omega_{h,i}^n$, and 
the quantities $\bu_h^n$, $\bv_h^n$, $p_h^n$, $q_h^n$, and $\bX_\hs^n$. The
	definition of $\gamma_1,\gamma_2,\gamma_{\bu}$ and $\gamma_p$ are given
	later in the paper; see \eqref{gam1 and gam2 def}.
	\begin{algorithm}[t!]
			\refstepcounter{problem}
			{\bf Problem \theproblem: Discrete-time finite element formulation}
			\label{Problem: FEM}
			\begin{enumerate}
				\begin{subequations}
					\item Solve for $\bu^{n+1}_i$ and $p^{n+1}_i$ in
					\begin{alignat}{3}
					&\big<\bF^{n+1},\bv\big>= \kappa\sum_{j=0}^{m-1} \left(\partials{\bX_{j+1}^n}{s} - \partials{\bX_j^n}{s} \right) \left\{\bv^{n+1}(\bX^{n}_{j+1})\right\}
					\label{FEM: RHS function}\\
					\begin{split} \label{FEM: stokes}
					&\sum_{i=1}^2 \frac{1}{\Delta t}(\bu_i^{n+1}-\tilde\bu_i^n,\bv_i)_{\Omega_i^{n}} + \mu(\eps(\bu_i^{n+1}),\eps(\bv_i))_{\Omega_{i}^{n}} - (p_i^{n+1}, \nabla \cdot \bv_i)_{\Omega_{i}^{n}} \\
					&\quad  - \big(\{(\mu\eps(\bu^{n+1})-p^{n+1})\bnhat\},\jump{\bv}\big)_{\Gamma_\hs^{n}} - \big(\jump{\bu^{n+1}},\{\mu\eps(\bv)\bnhat\}\big)_{\Gamma_\hs^{n}}
					\\
					&\quad   - \big((\mu\eps(\bu_1^{n+1})-p_1^{n+1})\bnhat,\bv_1\big)_{\partial\Omega}- \big(\bu_1^{n+1},\mu\eps(\bv_1)\bnhat\big)_{\partial\Omega}
					\\
					&\quad  + \gamma_{\bu} \bjj_{i,h}^n(\bu_i^{n+1},\bv_i) +   \gamma_1 \frac{\mu}{h} ( \jump{\bu^{n+1}},\jump{\bv})_{\Gamma_\hs^n} +  \gamma_1 \frac{\mu}{h} ( \bu_1^{n+1}, \bv_1)_{\partial\Omega}
					\\
					&   \quad  + \gamma_2\frac{h}{\Delta t} ( \jump{\bu^{n+1}\cdot\bnhat},\jump{\bv \cdot\bnhat})_{\Gamma_\hs^n} + \gamma_2\frac{h}{\Delta t} ( \bu_1^{n+1}\cdot\bnhat , \bv_1\cdot \bnhat)_{\partial\Omega}  
					\\
					& \quad + \nu\kappa\Delta t \int_{\Gamma_{\hs}}^n \partials{\bu^{n+1}(\bX^n)}{s}\partials{}{s}\avg{\bv(\bX^n(s))}\, ds = \ip{\bF^{n+1},\bv}
					\end{split}
					\\
					\begin{split} \label{FEM: incompressibility}
					&\sum_{i=1}^2 -(\nabla\cdot \bu_i^{n+1}, q_i)_{\Omega_{i}^{n}} + \big(\jump{\bu^{n+1}},\{q\bnhat\}\big)_{\Gamma_\hs^{n}} + (\bu_1^{n+1},q_1\bnhat)_{\partial\Omega}\\
					&\qquad- \gamma_p J_{i,h}^n(p_i^{n+1},q_i) = 0
					\end{split}
					\end{alignat}
					for all $\bv_i^{n+1} \in \bV_{i,h}^{n}$ and $q_i^{n+1} \in M_{i,h}^{n}$, $i= 1, 2$.
					
					\item Solve for $\bX_j^{n+1}$ using
					\begin{alignat}{2}
					\bX^{n+1}_j = \bX^n_j  + \Delta t \{\bu^{n+1}(\bX^{n}_j)\} \qquad \text{ for } j = 0,\dots, m.
					\label{FEM: discrete time boundary movement}
					\end{alignat}
				\end{subequations}
			\end{enumerate}
		
	\end{algorithm}

	\begin{remark}
		To distinguish between the explicit and semi-implicit methods in Problem \ref{Problem: FEM}, we define a parameter $\nu$ which can be set to either 0 or 1.
		Setting $\nu=1$ yields the semi-implicit method, while $\nu=0$ leaves us with the explicit method.
	\end{remark}

\subsection{Energy stability of the FEM}

We are able to prove the unconditional stability of the semi-implicit method in Problem \ref{Problem: FEM} under the assumption below, similar to that seen in \cite{Massing14}.\\

\noindent {\bf Assumption 1.} Given $K\in \Tc_h^{n,\Gamma}$,  there exists $K'\in \Tc_{i,h}^{n,I}$,  an integer $N>0$, and $N$ elements $\{K_k\}_{k=1}^N$ such that $K_1 = K$, $K_N=K'$ and $K_k\cap K_{k+1} \subset \overline{\Omega_{1,h}^{n,e}\cap \Omega_{2,h}^{n,e}}$.
\\

The next lemma is necessary to bound the strain on the extended subdomain $\Omega_{i,h}^{n,e}$ by the strain on the original subdomain $\Omega_{i,\hs}^n$.
The result shows why it is necessary to include $j_{i,h}(\bu,\bv)$ for stability.

\begin{lemma}
	\label{Lemma: extended bound}
	Suppose $\bv \in \bV_{i,h}^n$ is continuous on $\Omega_{i,h}^{n,e}$ and Assumption 1 
	holds for $\Tc_{i,h}^n$.
	Then we have the following estimate:
	\begin{equation*}
	\norm{\eps(\bv)}_{L^2(\Omega^{n,e}_{i,h})}^2 \leq C_{\eps}\left( \norm{\eps(\bv)}_{L^2(\Omega_{i,\hs}^n)}^2 + \bjj_{i,h}(\bv,\bv)  \right),
	\end{equation*}
	where $C_{\eps}$  depends on neither $\bv$ nor $h$.
	\label{DSDT: extended omega estimate}
\end{lemma}

\begin{proof}
	Let $K_1$ and $K_2$ be neighboring square elements with a shared edge $F = K_1\cap K_2$.
	We make use of the following result from Lemma 5.1 in \cite{Massing14}:
	\begin{equation}
	\norm{v}_{L^2(K_1)}^2 \leq C_m\left(\norm{v}_{L^2(K_2)}^2 + \sum_{0\leq\ell\leq p} h^{2\ell+1} \int_F [\partial_{\bnhat_F}^\ell v]^2 \right),
	\label{Massing inequality}
	\end{equation} 
	where $v|_{K_1}$, $v|_{K_2}$  are polynomial functions of degree less than or equal to $p$.
	With $\bV_{i,h}^n$ defined as above, the summation in \eqref{Massing inequality} simplifies to $p=1$.
	Denoting $\bv = (u,v)$, we apply the inequality \eqref{Massing inequality} directly to each term of $\eps(\bv):\eps(\bv) = (\partial_x u)^2 + \frac{1}{2}(\partial_x v + \partial_y u)^2 + (\partial_y v)^2$ and add the inequalities.
	Note that on any vertical edge the jump of all $y$ derivatives will be zero because $\bv$ is continuous across $F$ and $\bv|_F$ is simply a polynomial in each component.
	Also on a vertical edge the unit normal vector is $\bnhat_F = (1,0)$ and it follows that
	$$[\partial_{\bnhat_F}^\ell (\eps(\bv)\bnhat_F)]\cdot [\partial_{\bnhat_F}^\ell (\eps(\bv)\bnhat_F)] = ([\partial_x^\ell \partial_x u])^2 + \frac{1}{4}([\partial_x^\ell (\partial_x v + \partial_y u)])^2.$$
	The resulting inequality is
	\begin{align*}
	&\norm{\eps(\bv)}_{L^2(K_1)}^2 = \int_{K_1} (\partial_x u)^2 + \frac{1}{2}(\partial_x v + \partial_y u)^2 + (\partial_y v)^2 \\
	& \qquad \quad \leq C_m\bigg(\int_{K_2} (\partial_x u)^2 + \frac{1}{2}(\partial_x v + \partial_y u)^2 + (\partial_y v)^2\\
	&\qquad \quad \qquad \qquad + \sum_{\ell=0}^1 h^{2\ell+1}\int_F [\partial_{\bnhat_F}^\ell \partial_x u]^2 + \frac{1}{2}[\partial_{\bnhat_F}^\ell (\partial_x v + \partial_y u)]^2 \bigg)\\
	&\qquad \quad \leq 2C_m\left(\norm{\eps(\bv)}_{L^2(K_2)}^2 + \sum_{\ell=0}^1 h^{2\ell+1}\int_F [\partial_{\bnhat_F}^\ell (\eps(\bv)\bnhat_F)]^2\right).
	\end{align*}
	Similarly for any horizontal edge, we let $\bnhat_{F} = (0,1)$ and use the fact that $[\partial_x u]=0$ to get the same inequality.
	
	Using Assumption 1, we are able to find a sequence of at most $N$ adjacent elements leading from an element $K_1\in \Tc_h^{n,\Gamma}$ to an element $K_N \in \Tc_{i,h}^{n,I}$.
	Applying the above inequality across each of the edges $F_k = K_k\cap K_{k+1}$, for $1\leq k\leq N-1$, we have
	\begin{equation}
	\norm{\eps(\bv)}_{L^2(K_1)}^2 \leq (2C_m)^N\left( \norm{\eps(\bv)}_{L^2(K_N)}^2 +  \sum_{k = 1}^N \sum_{\ell=0}^1 \int_{F_k} [\partial_{\bnhat_{F_k}}^\ell \eps(\bv)\bnhat_{F_k}]^2\right).
	\label{neighbor bound}
	\end{equation}
	Repeating \eqref{neighbor bound} for all $K\in \Tc_h^{n,\Gamma}$  and denoting $C_{\eps} = (2C_m)^N$ completes the proof.
\end{proof}

Recall the definition of the norm over the reference configuration $R = [0,L]$.
For a quantity $\phi(s)$ that is constant on each linear segment of the polygonal $\Gamma_\hs^n$, we can simplify \eqref{R norm} to
$$\normgam{\phi}^2 = \sum_{j=0}^{m-1} \left(\phi(s_j)\right)^2\, (s_{j+1}-s_j).$$
Above $\phi(s_j)$ denotes the value of $\phi(s)$ on the interval $(s_j,s_{j+1})$.
Using the definition of the energy \eqref{energy definition} we are able to prove the following theorem.

\begin{theorem}\label{DSDT: semi-implicit theorem}
	Let $\bu_{i,h}^{n+1}$, $p_{i,h}^{n+1}$, and $\bX_\hs^{n+1}$ be solutions to \eqref{FEM: RHS function}-\eqref{FEM: discrete time boundary movement} at time $t^{n+1}$ with $\nu = 1$ and assume Assumption 1 holds.  Then the following inequality holds:
	\begin{align}
	\begin{split}\label{DSDT: semi-implicit inequality}
	E^{n+1} \leq &E^n - \frac{1}{2}      \norm{\bu_h^{n+1}-\bu_h^{n}}^2_{L^2(\Omega)} + \frac{1}{2} 
	{\kappa}  \normgam{\partials{\bX_\hs^{n+1}}{s}-\partials{\bX_\hs^{n}}{s}}^2                                                               \\
	& + \Delta t \bigg(  \gamma_1\frac{\mu}{h}\int_{\Gamma_\hs^n} \jump{\bu_h^{n+1}}^2\, ds 
	+          \gamma_1\frac{\mu}{h} \int_{\partial\Omega} (\bu_{1,h}^{n+1})^2 \, ds
	\\ & +
	\gamma_2 \frac{h}{\Delta t} 
	\int_{\Gamma_\hs^n} \jump{\bu_h^{n+1}\cdot \bnhat }^2\, ds 
	+ \gamma_2 \frac{h}{\Delta t} 
	\int_{\partial\Omega} (\bu_{1,h}^{n+1}\cdot\bnhat)^2 \, ds
	\\
	& +          \sum_{i=1}^2 \left(\gamma_{\bu} \bjj_{i,h}(\bu_{i,h}^{n+1},\bu_{i,h}^{n+1})  + \gamma_p J_{i,h}(p_{i,h}^{n+1},p_{i,h}^{n+1})\right) \bigg).
	\end{split}
	\end{align}
\end{theorem}
\begin{proof}
	We first let $\bv_h = \bu_h^{n+1}$ and $q_h=p_h^{n+1}$ in \eqref{FEM: RHS function}-\eqref{FEM: incompressibility} and subtract \eqref{FEM: incompressibility} from \eqref{FEM: stokes}.
	After cancellation due to the symmetry of the $L^2$ inner product and some simplification we have
	\begin{align}
	\begin{split}
	\sum_{i=1}^2 &\frac{1}{\Delta t} \left(\bu^{n+1}_{i,h}-\bu^n_{i,h}, \bu^{n+1}_{i,h}\right)_{\Omega_{i,\hs}^n} + \sum_{i=1}^2\mu \norm{\eps(\bu_{i,h}^{n+1})}_{L^2(\Omega_{i,\hs}^n)} \\
	& - 2\left( \jump{\bu_h^{n+1}}, \{\mu\eps(\bu_h^{n+1})\bnhat\}\right)_{\Gamma_\hs^n} - 2\left( \bu_{1,h}^{n+1}, \mu\eps(\bu_{1,h}^{n+1})\bnhat_1\right)_{\partial \Omega} \\
	& \quad + \gamma_{\bu} \bjj_h(\bu_{i,h}^{n+1},\bu_{i,h}^{n+1}) + \gamma_p J_h(p_{i,h}^{n+1},p_{i,h}^{n+1})\\
	&+ \gamma_1\frac{\mu}{h}\int_{\Gamma_\hs^n} \jump{\bu_h^{n+1}}^2 + \gamma_1 \frac{\mu}{h} \int_{\partial\Omega} (\bu_{1,h}^{n+1})^2\\
	&+ \gamma_2\frac{h}{\Delta t}\int_{\Gamma_\hs^n} \jump{\bu_h^{n+1}}^2 + \gamma_2 \frac{h}{\Delta t} \int_{\partial\Omega} (\bu_{1,h}^{n+1})^2\\  
	&= \kappa\sum_{j=0}^{m-1} \left(\partials{\bX^{n+1}_{j+1}}{s} - \partials{\bX^{n+1}_j}{s} \right) \left\{\bu_h^{n+1}(\bX^{n}_{j+1})\right\}.
	\end{split}
	\label{BE energy thing}
	\end{align}
		The term on the right-hand side of \eqref{BE energy thing} is obtained using \eqref{force summation} and \eqref{FEM: discrete time boundary movement} as follows:
		\begin{align*}
		\kappa\sum_{j=0}^{m-1}& \left(\partials{\bX^{n+1}_{j+1}}{s} - \partials{\bX^{n+1}_j}{s} \right) \left\{\bu_h^{n+1}(\bX^{n}_{j+1})\right\}\\
		& =
		-\kappa\int_0^L \partials{\bX^{n+1}}{s} \partials{}{s}\left\{\bu_h^{n+1}(\bX_\hs^{n}(s))\right\}\, ds\\
		&= -\kappa\int_0^L \partials{\bX^{n}}{s} \partials{}{s}\left\{\bu_h^{n+1}(\bX_\hs^{n}(s))\right\}\, ds\\
		&\qquad\qquad 
		- \kappa\Delta t \int_0^L \avg{\partials{\bu_h^{n+1}(\bX^n)}{s}} \cdot \avg{\partials{\bu_h^{n+1}(\bX^n)}{s}}\, ds\\
		&=
		\kappa\sum_{j=0}^{m-1} \left(\partials{\bX^n_{j+1}}{s} - \partials{\bX^n_j}{s} \right) \left\{\bu_h^{n+1}(\bX^{n}_{j+1})\right\}\\
		&\qquad \qquad
		- \frac{\kappa\Delta t}{2} \int_0^L \avg{\partials{\bu_h^{n+1}(\bX^n)}{s}} \cdot \avg{\partials{\bu_h^{n+1}(\bX^n)}{s}}\, ds.
		\end{align*}

	Following the same simplification as in Lemma \ref{Dynamic: semi-implicit theorem} we have
	$$\sum_{i=1}^2 \left(\bu^{n+1}_{i,h}-\bu^n_{i,h}, \bu^{n+1}_{i,h}\right)_{\Omega_{i,\hs}^n}  = \frac{1}{2}\left( \norm{\bu_h^{n+1}}_{L^2(\Omega)}^2 + \norm{\bu_h^{n+1} - \bu_h^n}_{L^2(\Omega)}^2 - \norm{\bu_h^{n}}_{L^2(\Omega)}^2\right)$$
	and
	\begin{align*}
	\kappa\sum_{j=0}^{m-1} &\left(\partials{\bX^{n+1}_{j+1}}{s} - \partials{\bX^{n+1}_j}{s} \right) \left\{\bu_h^{n+1}(\bX^{n}_{j+1})\right\} \\ 
	& \qquad=\frac{\kappa}{2\Delta t}\left( \normgam{\partials{\bX_\hs^{n+1}}{s}}^2 + \normgam{\partials{\bX_\hs^{n+1}}{s} - \partials{\bX_\hs^n}{s}}^2 - \normgam{\partials{\bX_\hs^{n}}{s}}^2\right).
	\end{align*}
	Now we look to control the integrals over the boundaries.
	First, using the Cauchy-Schwarz inequality and a generalized inequality of arithmetic-geometric means we have for any $\gamma_1>0$
	\begin{align*}
	&2\left( \jump{\bu_h^{n+1}}, \{\mu\eps(\bu_h^{n+1})\bnhat\}\right)_{\Gamma_\hs^n}
	\leq \frac{\gamma_1\mu}{h}\norm{\jump{\bu_h^{n+1}}}_{L^2(\Gamma_\hs^n)}^2 + \frac{h\mu}{\gamma_1}\norm{\avg{\eps(\bu_{h}^{n+1})}}_{L^2(K\cap\Gamma_\hs^n)}^2\\
	& \qquad \qquad \qquad  \qquad = \frac{\gamma_1\mu}{h}\norm{\jump{\bu_h^{n+1}}}_{L^2(\Gamma_\hs^n)}^2  + \sum_{i=1}^2 \sum_{K\in \Tc_{i,h}^n}\frac{h\mu}{\gamma_1}\norm{\eps(\bu_{i,h}^{n+1})}_{L^2(K\cap\Gamma_\hs^n)}^2.
	\end{align*}
	Now we look to bound the norm of the average of the symmetric gradient over the interface, which has been separated into an interior and exterior component using the triangle inequality.
	For some function $v\in H^1(K)$, with the help of Lemma 1 in \cite{Guzman17} noting that the polygonal interface $\Gamma_\hs^n$ is Lipschitz,
	\begin{align}
	\norm{v}_{L^2(K\cap \Gamma_\hs^n)}^2 &\leq C_1\left( h^{-1}\norm{v}_{L^2( K)}^2 + h\norm{\nabla v}_{L^2( K)}^2 \right)\nonumber\\
	&\leq C_1 h^{-1}\norm{v}_{L^2( K)}^2 + C_1\cdot \tilde{C}_I h^{-1} \norm{v}_{L^2( K)}^2.\label{DSDT: interface bound}
	\end{align}
	Here, $C_1$ is the constant from \cite{Guzman17} and $\tilde{C}_I$ is the constant from the well-known finite element inverse inequality
	$$\norm{\nabla v}_{L^2( K)}^2 \leq \tilde{C}_I h^{-2} \norm{v}_{L^2(K)}^2.$$
	Letting $v$ be each component of the symmetric part of the gradient in \eqref{DSDT: interface bound} and adding the resulting inequalities and denote
	${C}_I = \tilde{C}_I +1 $ yields
	$$\norm{\eps(\bu^{n+1}_{i,h})}_{L^2(K\cap\Gamma_\hs^n)}^2 \leq C_1 {C}_I h^{-1} \norm{\eps(\bu^{n+1}_{i,h})}_{L^2(K)}^2$$
	and we can control the inner product over $\Gamma_\hs^n$ by
	\begin{align}
	\left( \jump{\bu_h^{n+1}}, \{\mu\eps(\bu_h^{n+1})\}\right)_{\Gamma_\hs^n}
	&\leq
	\frac{\gamma_1 \mu}{h} \norm{\jump{\bu_h^{n+1}}}_{L^2(\Gamma_\hs^n)}^2
	+ \frac{C_1 C_I}{\gamma_1} \sum_{i=1}^2 \sum_{K\in \Tc_h} \mu \norm{\eps(\bu_{i,h}^{n+1})}_{L^2(K)}^2 \nonumber\\
	&= \frac{\gamma_1 \mu}{h} \norm{\jump{\bu_h^{n+1}}}_{L^2(\Gamma_\hs^n)}^2
	+  \frac{C_1 C_I}{\gamma_1} \sum_{i=1}^2 \mu\norm{\eps(\bu_{i,h}^{n+1})}_{L^2(\Omega_{i,\hs}^{n,e})}^2 \nonumber \\
	&\leq \frac{\gamma_1 \mu}{h} \norm{\jump{\bu_h^{n+1}}}_{L^2(\Gamma_\hs^n)}^2
	+ \frac{ C_1 C_I C_{\eps}}{\gamma_1} \bigg( \mu\norm{\eps(\bu_h^{n+1})}_{L^2(\Omega)}^2\nonumber\\
	&\quad \qquad + \sum_{i=1}^2 \mu j_{i,h}(\bu_{i,h}^{n+1},\bu_{i,h}^{n+1})\bigg),\nonumber
	\end{align}
	where the final inequality is the result of Lemma \ref{Lemma: extended bound}.
	Similarly, we have
	\begin{equation*}
	2 \left( \bu_{1,h}^{n+1}, \mu\eps(\bu_{1,h}^{n+1})\right)_{\partial \Omega} \leq \frac{\gamma_1 \mu}{h} \norm{\bu_{1,h}^{n+1}}_{L^2(\partial\Omega)}^2 + \frac{ C_1 C_I C_{\eps}}{\gamma_1} \mu \norm{\eps(\bu_h^{n+1})}_{L^2(\Omega)}^2.
	\end{equation*}
	Now we combine all of these inequalities and choose $\gamma_1$
	so that  $\frac{2C_1 C_I C_{\eps}}{\gamma_1} \leq 1$ and multiply both sides by $\Delta t$ to get \eqref{DSDT: semi-implicit inequality}.
\end{proof}

	To establish the discrete inf-sup condition for Q2-P1 time-depedent Stokes elements
	for CutFEM, see \cite{Hansbo14,MR3544656, MR3807820, MR1950614, MR2092744, MR2421892}. This will be published elsewhere in the context of the CutFEM shown in this work. We
	note however that  to establish energy stability
	we only need to control $(\jump{\bu^{n+1}}, \{q^{n+1}\}\bnhat)_{\Gamma_\hs^{n}} = (\jump{\bu^{n+1}}\cdot\bnhat, \{q^{n+1}\})_{\Gamma_\hs^{n}}$. Using similar arguments as in the sources above we have: \\
	\begin{enumerate}
	\item  If $\mu/h$ dominates $h/{\Delta t}$
	we control via 
	\begin{equation*}
	(\jump{\bu^{n+1}}\cdot\bnhat, \{q^{n+1}\})_{\Gamma_\hs^{n}} \leq 
	\frac{\gamma_1\mu}{h} \norm{\jump{\bu_h^{n+1}}}_{L^2(\Gamma_\hs^n)}^2
	+ \frac{h}{4\gamma_1 \mu} \norm{\{q^{n+1}\}}_{L^2(\Gamma_\hs^n)}^2,
	\end{equation*}
	\item If $h/{\Delta t}$ dominates  $\mu/h$ we do
	\begin{equation*}
	(\jump{\bu^{n+1}}\cdot\bnhat, \{q^{n+1}\})_{\Gamma_\hs^{n}} \leq 
	\frac{\gamma_2 h}{{\Delta t}} \norm{\jump{\bu_h^{n+1}\cdot\bnhat}}_{L^2(\Gamma_\hs^n)}^2
	+ \frac{\Delta t}{4\gamma_2 h} \norm{\{q^{n+1}\}}_{L^2(\Gamma_\hs^n)}^2.
	\end{equation*}
	
	\end{enumerate}
	Hence, let us define 
	\begin{equation}
	\gamma_p = \min\left\{\frac{1}{4\gamma_1\mu},\frac{\Delta t}{4\gamma_2h^2}\right\}\quad \text{ and } \quad  \gamma_{\bu} =  \gamma_1 \mu + \gamma_2 \frac{h^2}{\Delta t}.
	\label{gam1 and gam2 def}
	\end{equation}


\section{Numerical results}
\label{Sec: Numerical results}

Below we illustrate the theoretical findings and some
approximation results with numerical simulations.
The numerical test cases confirm the unconditionally energy stability
of the semi-implicit method and conditionally stability of the
explicit method. 

In each example we choose the computational domain to be the square $\Omega = (0,1)^2$ with  the fluid initially at rest.
Let the reference configuration for $\Gamma$ be the unit interval, i.e., $L=1$.
We subdivide the interval $[0,1]$ into $m+1$  equally-spaced points $s_j$ 
such that the $\bX_j^0 := \Gamma(s_j,0)$ satisfies $\max_j |\bX_{j+1}^0 - \bX_j^0| < h/2$.
Thus, the step size in $[0,1]$ is $\hs = 1/m$.
We choose the penalty parameters from \eqref{jump of u} and \eqref{jump of u square} to be $\gamma_1=\gamma_2 = 10$.  With our choice of uniform $\hs$ we further expect $\Gamma_\hs^n$ to approach a regular polygon with the sampled points equally spaced along the interface.

In each example, since the coupled problem can be reduced to
a second-order in time partial differential equation,
the immersed boundary should oscillate and due to
the viscosity it should converge to a circular steady state. Due to
the incompressibility of the fluid, the interior area enclosed by
the membrane should not change in time. Another goal of the numerical
tests is to show good approximation for such properties.

\subsection{Example 1: Spatial convergence}

The results in Table \ref{convergence table} illustrate the convergence of our method in the steady-state problem
\begin{align}
\begin{split} \label{static problem}
-\mu\nabla\cdot\eps(\bu_i) + \nabla p_i &= {\bf f} \qquad \text{ in } \Omega_i, \, i = 1, 2 ,\\
\nabla\cdot \bu_i &= 0 \qquad  \text{ in } \Omega_i, \, i = 1,2.
\end{split}
\end{align}
The boundary condition for $\bu_1$ on $\partial \Omega$ and jump conditions $\jump{(\mu \eps(\bu) - p)\bnhat}$, $\jump{\bu\cdot\bnhat}$, and $\jump{\bu}$ on $\Gamma$ are chosen to match the exact test solutions
\begin{align}
\begin{split} \label{static solution}
\bu_1 = \pmat{\sin(x)\cos(y)\\ -\cos(x)\sin(y)}&, \quad p_1 = \sin(2\pi x) \cos(2\pi y),\\
\bu_2 = \pmat{x e^{-xy}\\ -y e^{-xy}}&, \quad p_2 = x^2 y^2.
\end{split}
\end{align}

\begin{table}[t]
	\caption[Error approximations illustrating convergence in the static test problem.]{Error tables for the (a) velocity  and (b) pressure  solving \eqref{static problem} in Example 1.
	The numbers shown are the values of the difference between the approximate solution and the exact solution given by \eqref{static solution} in the specified norm with spatial grid size $h$.
	Each rate $k$ corresponds to the convergence rate $O(h^k)$.
		\label{convergence table}}
		\begin{subtable}{\textwidth}
			\caption{Error of the velocity.
				\label{velocity convergence table}}
			\begin{center}
				\scalebox{0.9}{
					\begin{tabular}{  l | c  c | c c | c c | c c} 
						$1/h$ & $L^2$ & $k$ & $H^1$ & $k$ & $L^\infty$ & $k$ & $W^{1,\infty}$ & $k$ \\ 
						\hline
						8 & 6.0751e-4 & &  1.2079e-2  & & 2.9443e-3 & &  1.5822e-1
						\\ 
						16 & 5.7992e-5 & 3.4 &  2.8194e-3 & 2.1 &  4.0194e-4 & 2.9 &  5.0981e-2 & 1.6
						\\ 
						32 & 4.0155e-6 & 3.9 &  4.7479e-4 & 2.6 &  4.6180e-5 & 3.1 &  1.3094e-2 & 2.0
						\\
						64 & 3.8898e-7 & 3.4 &  8.4839e-5 & 2.5 &  7.5565e-6 & 2.6 &  3.7782e-3 & 1.8
						\\ 
						128 & 3.0663e-8  & 3.7 & 1.5375e-5 & 2.5 &  9.9287e-7 & 3.0 &  1.0938e-3 & 1.8 
						
					\end{tabular}
					\vspace{-15pt}
				}
			\end{center}
			
		\end{subtable}
		
		\begin{subtable}{\textwidth}
			\caption{Error of the pressure.
				\label{pressure convergence table}}
			\begin{center}
				\vspace{12pt}
				\scalebox{0.9}{
					\begin{tabular}{  r | c  c | l c | c c | c l} 
						$1/h$ & $L^2$ & $k$ & $H^1$ & $k$ & $L^\infty$ & $k$ & $W^{1,\infty}$ & $k$ \\ 
						\hline
						8 & 3.7455e-2 &  &  1.3695 &  &  2.3560e-1 &  &   6.5007 &
						\\ 
						16 & 7.1874e-3 & 2.4 &   6.1616e-1  & 1.2 & 5.2798e-2 & 2.2 &  2.9639 & 1.1
						\\ 
						32 & 1.7328e-3 & 2.1 &  3.0661e-1 & 1.0 &  1.6484e-2 & 1.7 &  2.0968 & 0.5
						\\
						64 & 4.1940e-4 & 2.0 &  1.5074e-1 & 1.0 &  5.4817e-3 & 1.6 &  1.3588 & 0.6
						\\ 
						128 & 1.0151e-4 & 2.0 &  7.4491e-2 & 1.0 &  1.4711e-3 & 1.9 &  7.4067e-1 &  0.9
						
					\end{tabular}
					\vspace{-40pt}
				}
			\end{center}
			
		\end{subtable}

\end{table}

The exact solution in \eqref{static solution} exhibits nonzero jumps in the velocity and stress across the interface, independent of our choice of $\Gamma$.
 Table \ref{convergence table} was generated using a circular interface of radius $r=0.3$ centered at $(0.5,0.5)$.
The discrete interface $\Gamma_\hs$ was constructed by choosing the reference configuration to be the unit interval, i.e., $L=1$, with $\hs=1/400$.
The $H^1$ error observed in these tables is optimal since we are using Q2-P1 elements.
Note that the rates $k$ seen in Table \ref{convergence table} correspond to the convergence rate $O(h^k)$.
We also see superconvergence in the $L^2$ and $H^1$ norms of the velocity and near-optimal convergence in the other norms.

To compute the $L^2$ and $H^1$ error, we extend $\bu$ to $\Omega_{i,\hs}$ when necessary using \eqref{static solution} and compute the norms of the difference $\bu_i-\bu_{i,h}$ on each subdomain.
The $L^\infty$ and $W^{1,\infty}$ norms are computed using the difference $\bu_i-\bu_{i,h}$ at all the nodes where the degrees of freedom of $\bu_h$ is
imposed, and also include the points on $\Gamma_\hs$, including each $s_j$ and all points where $\Gamma_\hs$ intersects edges of elements in $\Tc_{i,h}$ by interpolating $\bu_h$.

\subsection{Example 2: Ellipse}

The second example is a common scenario found in related literature \cite{Boffi08,Leveque97}.
The interface $\Gamma$ will begin as an ellipse where the initial points chosen are sampled from
\begin{equation}
\bX^0(s) = \pmat{0.3\cos(2\pi s) + 0.5\\ 0.4\sin(2\pi s) + 0.5}, \quad s\in [0,1].
\label{ellipse}
\end{equation}
The discrete interface $\Gamma_\hs^0$ is approximated by mapping $m+1$ equally-spaced points from $[0,1]$. 
It is worth noting that the lengths of two adjacent segments on $\Gamma_\hs^0$ may be different.
Since $\hs$ is chosen to be constant across each reference segment throughout all simulations, in addition to bending, the result is also a ``tension" force, or a stretching in the direction tangent to $\Gamma_\hs^0$.
The effects of such a force will be emphasized in Example 4.


\begin{figure}
	\begin{center}
		\makebox[\textwidth][c]{\includegraphics[width=\textwidth]{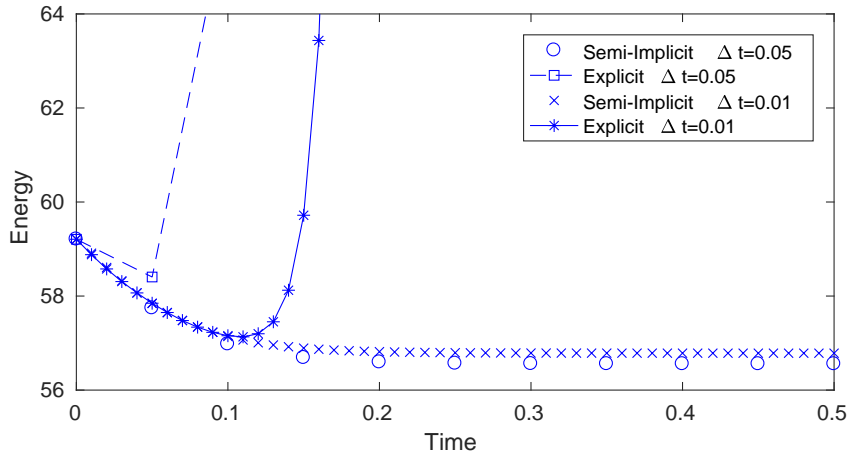}}
		\caption{Plot of the energy of the system in Example 2 comparing the explicit and semi-implicit methods with $\mu = 1$, $\kappa = 6$, $h=1/32$, and $\Delta t = 0.05$ ($\square$ for explicit, $\circ$ for semi-implicit) or $\Delta t = 0.01$ ($*$ for explicit, $\times$ for semi-implicit).
			The lines connecting data points are included when plotting the results of the explicit method to highlight the steep increase in energy.
		}
		\label{fig: ellipse implicit explicit comparison}
	\end{center}
\end{figure}

\begin{figure}	
	\begin{center}
		\includegraphics[width=.75\textwidth]{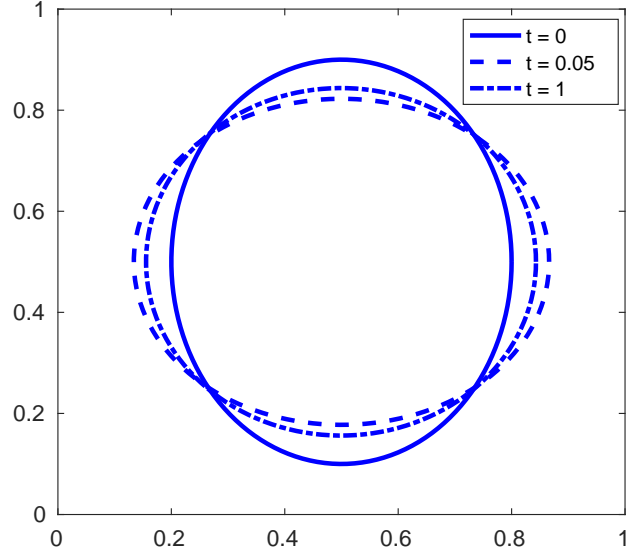} 
	\end{center}
	
	\caption{Plot of the position of the interface shown at $t=0$, $t=0.05$, and $t=1$.
		Simulation run with $\mu = 0.01$, $\kappa = 50$, $N=32$, $\Delta t = 0.01$, and the initial configuration given in Example 2.}
	\label{fig: ellipse three time steps}
\end{figure}

\begin{table}[]
	\caption{Normalized deviation of interior area at $t=0.5$ from initial interior area in Example 2.
		Results obtained using $\mu = 1$ and $\kappa = 10$ with $\Delta t$ and $h$ as shown.
	}
	\begin{center}
		\quad\qquad\qquad $\Delta t$
	\end{center}
	\centering
	\begin{tabular}{cr|cccc}
		& & 2.5e-3 & 1.25e-3 & 6.25e-4 & 3.125e-4\\
		\cline{2-6}
		& 16 & -4.3376e-04 & -2.1349e-04 & -1.0216e-04 & -4.7446e-05\\
		$1/h$ & 32 & -5.1587e-04 & -2.6936e-04 & -1.4444e-04 & -7.8798e-05\\
		& 64 & -5.0128e-04 & -2.5730e-04 & -1.3752e-04 & -7.6356e-05
	\end{tabular}
	
	\label{ellipse area table}
\end{table}

As seen in Figure \ref{fig: ellipse implicit explicit comparison}, the solution computed using the explicit method, with parameters chosen such that an
instability occurs, blows up very quickly as the energy fails to dissipate.
With a smaller time step, we see a more gradual increase in energy as the method does not fail so quickly.
The semi-implicit method exhibits the theoretical energy stability over the
explicit method and remains stable with each set of parameters tested.
Figure \ref{fig: ellipse three time steps} shows the position of the interface at three time steps capturing one intermediate step before steady state is achieved prior to $t=1$.

Table \ref{ellipse area table} shows the normalized deviation at time $t=0.5$ from the original interior area.
Due to the incompressibility of the fluid, the optimal result is a constant interior area as the interface moves.
Recall that $m$ is the number of points sampled from $\Gamma^0$ to form the polygon $\Gamma_\hs^0$.
As the mesh size $h$ decreases, we increase $m$.
In addition to improving the initial approximation of each subdomain, the conservation of the interior becomes more accurate as the mesh is refined.
We also see significant improvement in the conservation of interior area as $\Delta t$ is refined.
In Table \ref{ellipse area table} we see that the deviation from interior area is no larger than 0.05\% with the chosen parameters, and can be reduced to less than 0.008\% by refining $h$ and $\Delta t$.
It is worth noting that in this example a greater improvement is seen by reducing $\Delta t$ compared to reducing $h$.

We now turn to some observations of the temporal convergence of the semi-implicit method.
	In Table \ref{richardson table} the convergence of fluid velocity in the $L^2$ norm over the domain $\Omega$ is estimated using Richardson extrapolation.
	The value of the ratio shown in the table corresponds to
	\begin{equation}
	\label{Richardson ratio}
	r_{\Delta t} = \log_2\left(\frac{\norm{\bu_{\Delta t} - \bu_{\Delta t/2}}_{L^2(\Omega)}}{\norm{\bu_{\Delta t/2} - \bu_{\Delta t/4}}_{L^2(\Omega)}}\right),
	\end{equation}
	where $\bu_{\Delta t}$ is the approximation of $\bu$ at $t=0.1$ computed using time step $\Delta t$.
	The convergence of the method can be seen as $O({\Delta t}^{r})$. Since
	the method is high-order in $h$, the error is dominated by
	time discretization errors and Table \ref{richardson table} shows
	that the error of the fluid velocity in $L^2$ norm is asymptotically
	linear in $\Delta t$.

	Figure \ref{fig: ellipse traction} shows the point-wise linear convergence of the interior and exterior traction to $\Gamma_\hs$. 
	The traction is computed at the midpoint  of each segment of $\Gamma_\hs$, corresponding to $s_{j+\frac{1}{2}}$ in the reference configuration because the normal vector is not well-defined at each vertex of the polygon $\Gamma_\hs$. 
	In each plot the traction at $t=0.1$ computed using $\Delta t = 0.1/2^6$ is compared to the traction computed using $\Delta t = 0.1/2^3$, $\Delta t = 0.1/2^4$, and $\Delta t = 0.1/2^5$.
	Each line shows the sum of the absolute value of the difference in each component of the traction vector at each point $s_{j+\frac{1}{2}}$ in the reference configuration.
	Similarly, Figure \ref{fig: ellipse interface difference} shows the linear
	convergence in interface location.
	The interface location at $t=0.1$ computed using $\Delta t = 0.1/2^6$ is compared to the traction computed using $\Delta t = 0.1/2^3$, $\Delta t = 0.1/2^4$, and $\Delta t = 0.1/2^5$.
	The sum of the absolute value of the difference in each coordinate $\bX(s_j,0.1)$ is plotted.

	\begin{figure}
		
		\begin{center}
			\begin{subfigure}{\textwidth}
				\includegraphics[width=\textwidth]{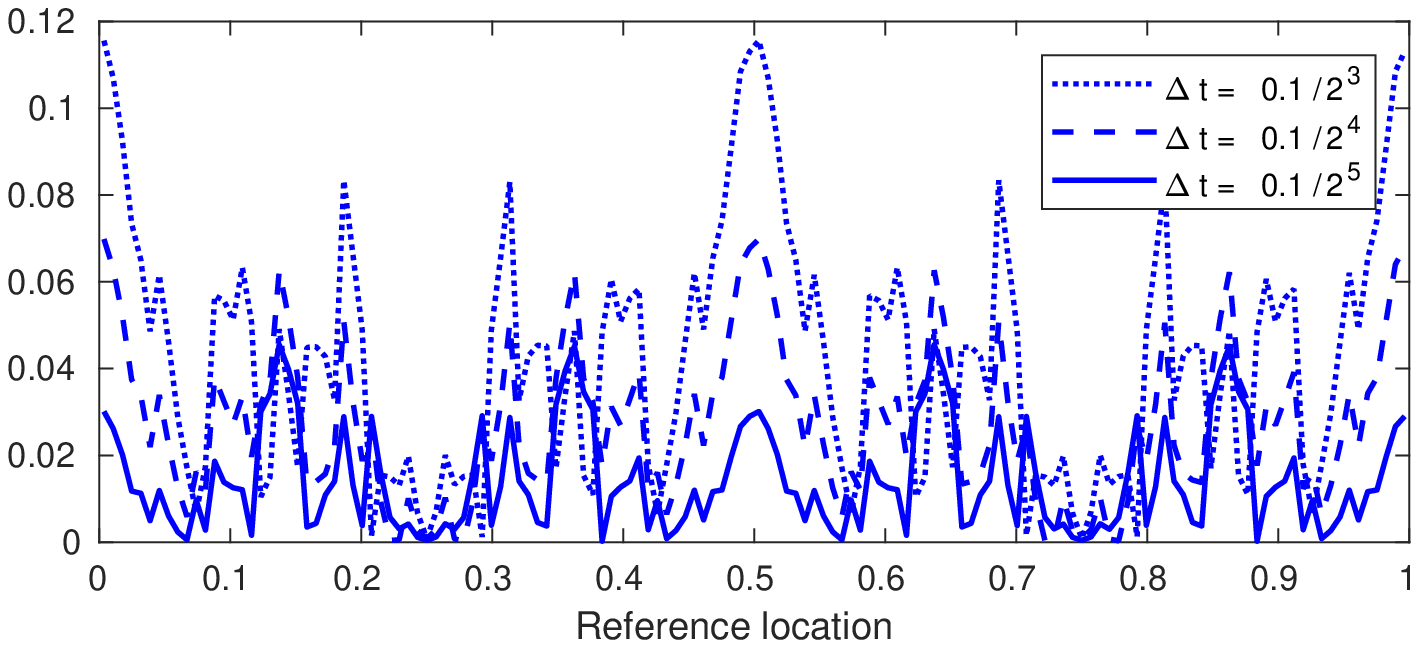}
				\caption{Interior traction.}
			\end{subfigure}
		\end{center}
		
		\begin{center}
			\begin{subfigure}{\textwidth}
				\includegraphics[width=\textwidth]{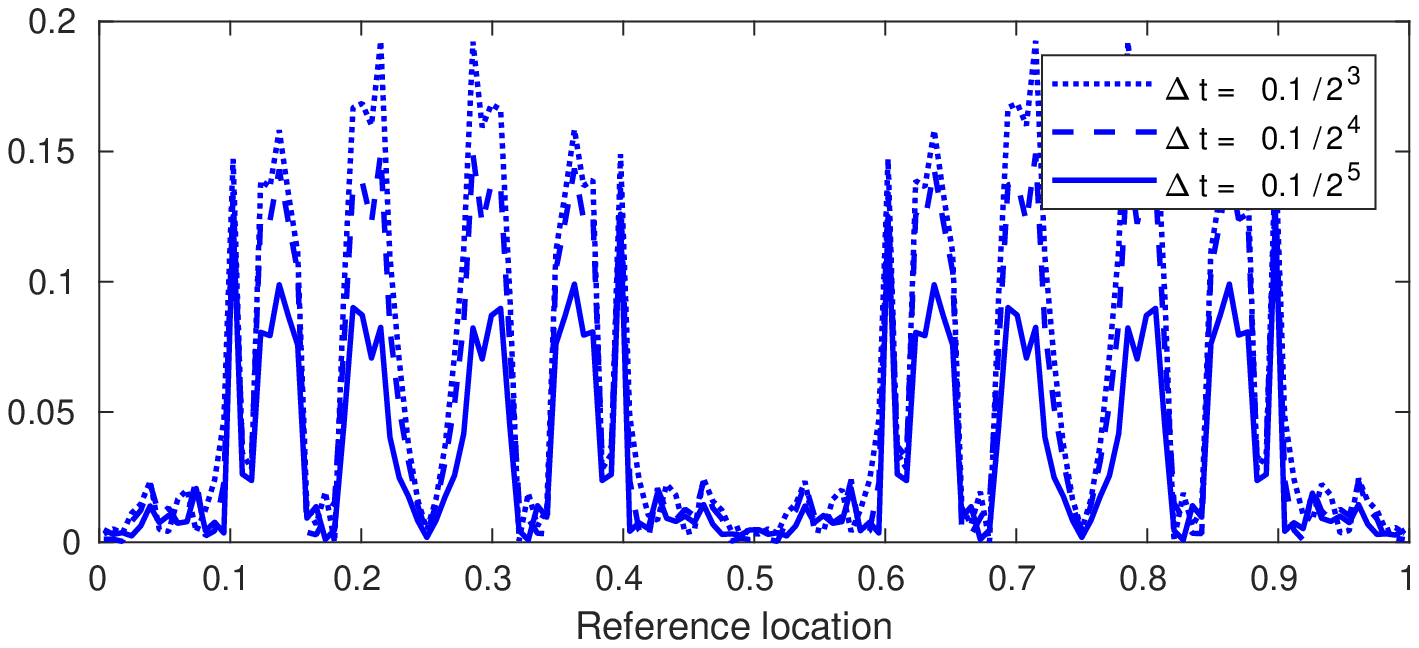}
				\caption{Exterior traction.}
			\end{subfigure}
		\end{center}
		
		\caption{The above figure shows the convergence of the traction at $t=0.1$ to the true solution computed using $\Delta t = 0.1/2^6$, $\mu=0.1$, and $\kappa=1$.
		The initial interface configuration is that described in Example 2.
		The quantity plotted is the sum of the absolute value of the difference at each midpoint on $\Gamma_\hs$.
		Point-wise convergence of the interior and exterior traction is observed.}
		\label{fig: ellipse traction}
	\end{figure}
	
	\begin{figure}	
		\begin{center}
			\includegraphics[width=\textwidth]{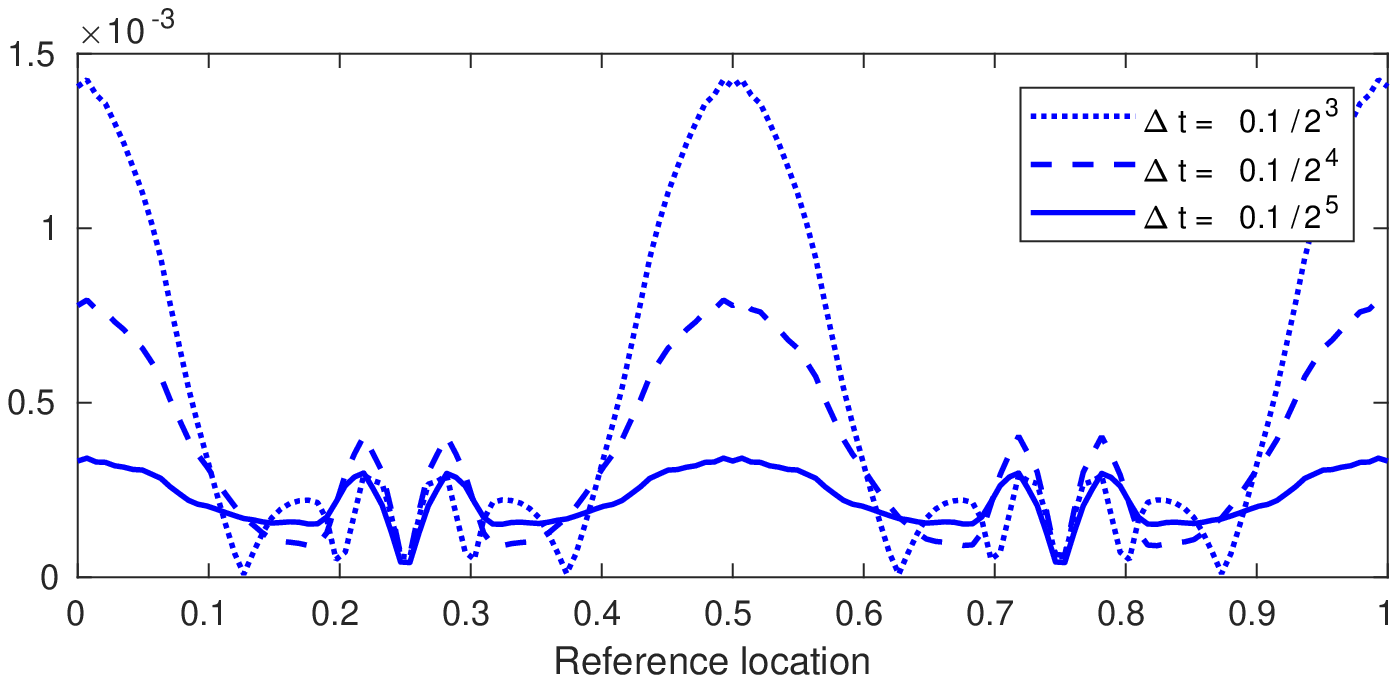} 
		\end{center}
		
		\caption{Plot of the difference between the interface location at $t=0.1$ computed using $\Delta t = 0.1/2^6$ and the time steps indicated in the plot with $\mu=0.1$ and $\kappa=1$.
		The initial interface configuration is that described in Example 2.
		The quantity plotted is the sum of the absolute value of the difference at each midpoint on $\Gamma_\hs$.}
		\label{fig: ellipse interface difference}
	\end{figure}

	\begin{table}[]
		\caption{Temporal convergence rates $r_{\Delta t}$ using Richardson extrapolation for Example 2 (ellipse) and Example 3 (heart) with $\mu=1$ and $\kappa=10$.
		The results in from this test indicate linear convergence of velocity $\bu$ in the $L^2$ norm when using the semi-implicit method.
		}
		\begin{center}
			\quad\qquad\quad $\Delta t$
		\end{center}
		\centering
		\vspace{5pt}
		\begin{tabular}{c|ccccc}
			 & 0.1& 0.1/2 & $0.1/2^2$ & $0.1/2^3$ & $0.1/2^4$\\
			\hline
			Ellipse &  1.5094 & 1.7405 & 1.9461 & 1.5505  &	-\\
			Heart & 3.7075 & 2.3262 & - & 4.7735 & 1.2742\\
		\end{tabular}
		
		\label{richardson table}
	\end{table}

\subsection{Example 3: Heart}

The third example is used to ensure that energy stability still holds regardless of the convexity of the interface and displacement of the centroid of the interior subdomain.
The original curve is constructed as the sum of two translated cardioids and is parameterized by $s\in [0,1]$ as follows:
\begin{equation*}
\bX^0(s) = \frac{1}{20}\pmat{\cos(2\pi s)\left(7(1-\sin(2\pi s)) + 3(1-\cos(2\pi s))\right) + 24\\ \sin(2\pi s)\left(3(1-\sin(2\pi s)) + 7(1-\cos(2\pi s))\right) + 24 }.
\end{equation*}
The energy plots in Figure \ref{fig: heart implicit explicit comparison} show that the energy in the explicit method becomes unstable slower than in the previous example.
However, the semi-implicit method remains stable.
Figure \ref{fig: heart three time steps} shows the position of the interface as it deforms and moves toward the top-right corner of $\Omega$, approaching a circular steady state.
In this figure we observe a quick deformation to a convex interior at $t=0.05$ and a translation of this region in the subsequent time steps.
Table \ref{heart area table} shows the normalized deviation from the original interior area at time $t=0.5$.
Contrary to the previous example, we see more improvement from reduction of $\Delta t$ than from refinement of $h$.
Here, the area loss is reduced almost linearly with the reduction in $\Delta t$ and very little corresponding to a smaller mesh size $h$.

\begin{figure}
	\begin{center}
		\makebox[\textwidth][c]{\includegraphics[width=\textwidth]{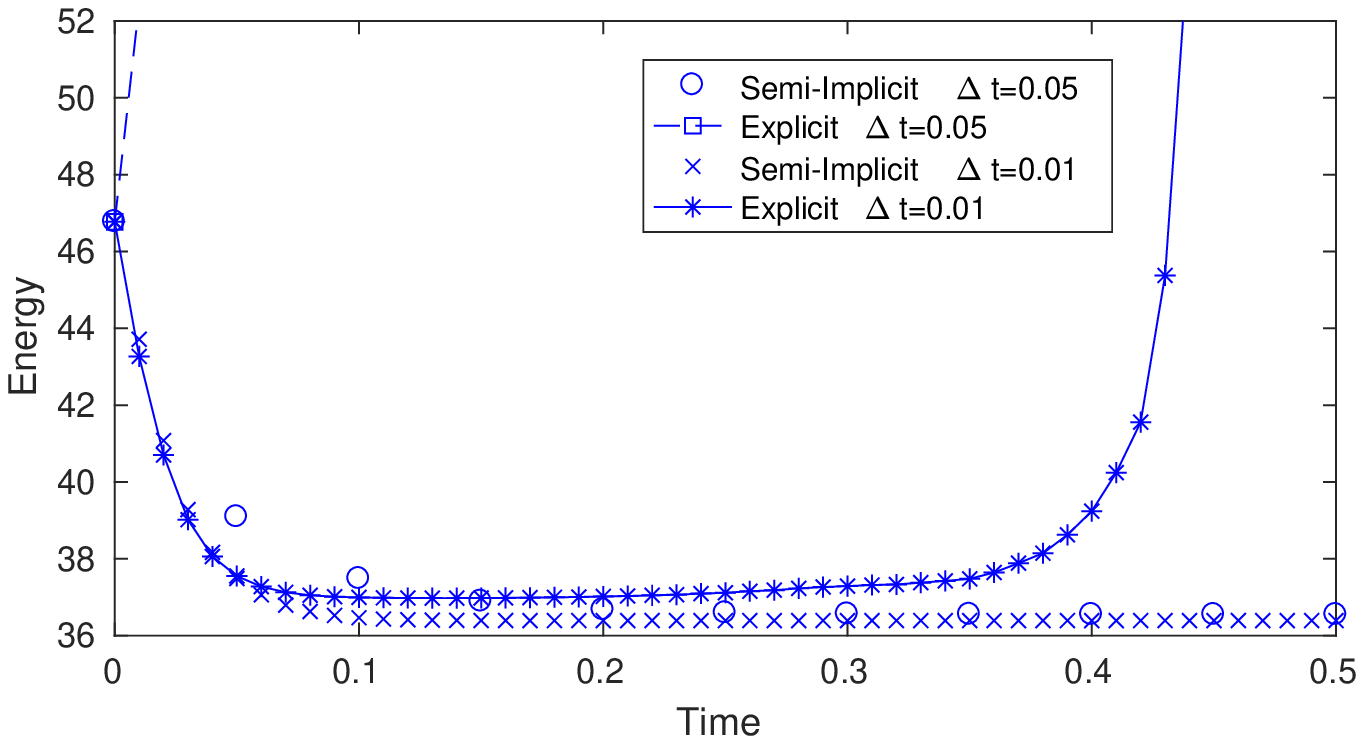}}
		\caption{ Plot of the energy of the system in Example 3,
			comparing the explicit and semi-implicit methods with $\mu = 1$, $\kappa = 6$, $N=32$, and $\Delta t = 0.05$ ($\square$ for explicit, $\circ$ for semi-implicit) or $\Delta t = 0.01$ ($*$ for explicit, $\times$ for semi-implicit).
			The lines connecting data points are included when plotting the results of the explicit method to highlight the steep increase in energy.
		}
		\label{fig: heart implicit explicit comparison}
	\end{center}
	
\end{figure}

\begin{figure}
	\centering
	\includegraphics[width=.75\textwidth]{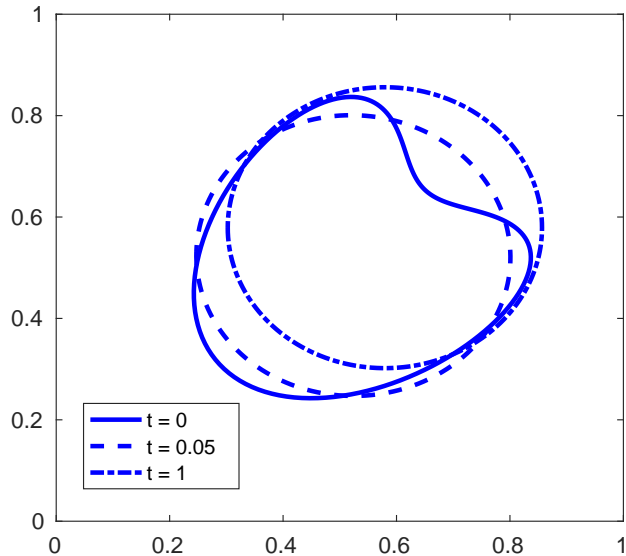}
	
	\caption{Plot of the position of the interface shown at $t=0$, $t=0.05$, and $t=1$.
		Simulation run with $\mu = 0.01$, $\kappa = 50$, $N=32$, $\Delta t = 0.01$, and initial configuration from Example 3.
	}
	\label{fig: heart three time steps}
\end{figure}

\begin{table}[]
	\caption{Normalized deviation of interior area at $t=0.5$ from initial interior area in Example 3.
		Results obtained using $\mu = 1$ and $\kappa = 10$ with $\Delta t$ and $h$ as shown.
	}
	\begin{center}
		\quad\qquad\qquad $\Delta t$
	\end{center}
	\centering
	\begin{tabular}{cr|cccc}
		& & 2.5e-3 & 1.25e-3 & 6.25e-4 & 3.125e-4\\
		\cline{2-6}
		& 16 & 4.3740e-03 &  2.3318e-03 &  1.2818e-03 &  7.7174e-04\\
		$1/h$ & 32 & 5.1529e-03  & 2.8014e-03  & 1.4961e-03  & 8.0600e-04\\
		& 64 & 5.0489e-03 &  2.7171e-03 &  1.4206e-03 &  7.3541e-04
	\end{tabular}
	
	\label{heart area table}
\end{table}

	In Table \ref{richardson table} the temporal convergence of the velocity in the $L^2$ norm is estimated using Richardson extrapolation alongside the rate of convergence estimates for the previous example.
	One can see that the rate of convergence is super-linear both examples.
	
	Figure \ref{fig: heart traction} shows the point-wise convergence of the interior and exterior traction.
	The traction is evaluated at the midpoint of each segment of $\Gamma_\hs$ and the quantity plotted is analogous to that described in the previous example.
	Figure \ref{fig: ellipse interface difference} shows the convergence of the interface at $t=0.1$ using refined values of $\Delta t$ compared to the interface location at $t=0.1$ using $\Delta t = 0.1/2^6$.
	
	\begin{figure}	
		\begin{center}
			\begin{subfigure}{\textwidth}
				\includegraphics[width=\textwidth]{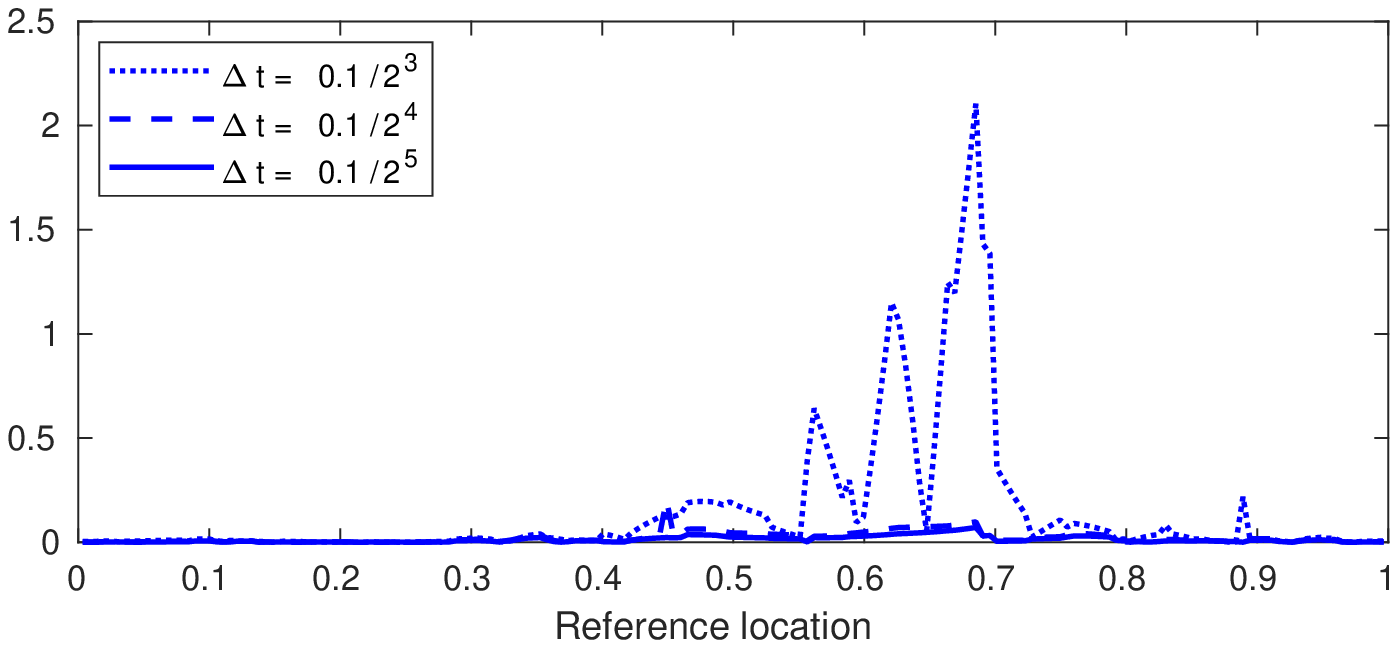} 
				\caption{Interior traction.}
			\end{subfigure}
			
		\end{center}
		
		\begin{center}
			\begin{subfigure}{\textwidth}
				\includegraphics[width=\textwidth]{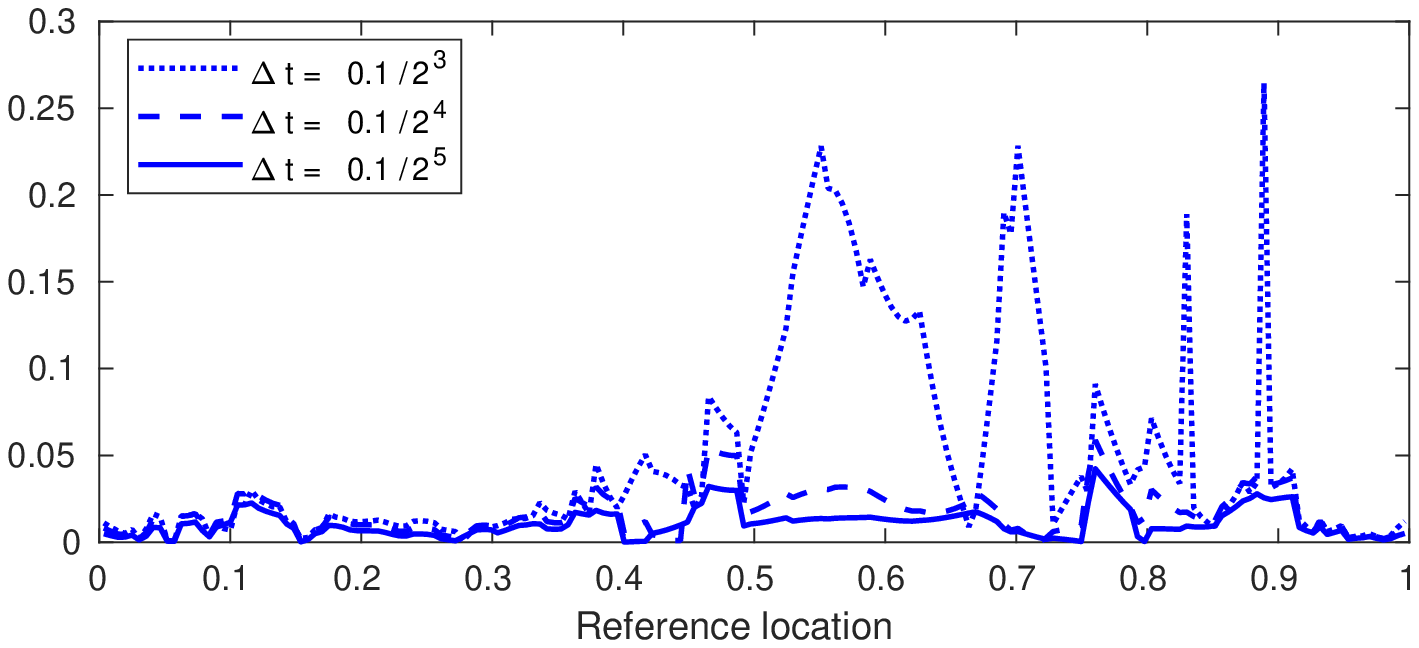}
				\caption{Exterior traction.}
			\end{subfigure}
			
		\end{center}
		
		\caption{The above figure shows the convergence of the traction at $t=0.1$ to the true solution computed using $\Delta t = 0.1/2^6$, $\mu=10$, and $\kappa=1$.
			The initial interface configuration is that described in Example 3.
			The quantity plotted is the sum of the absolute value of the difference at each point midpoint on $\Gamma_\hs$.
			Point-wise convergence of the interior and exterior traction is observed.}
		\label{fig: heart traction}
	\end{figure}
	
	\begin{figure}	
		\begin{center}
			\includegraphics[width=\textwidth]{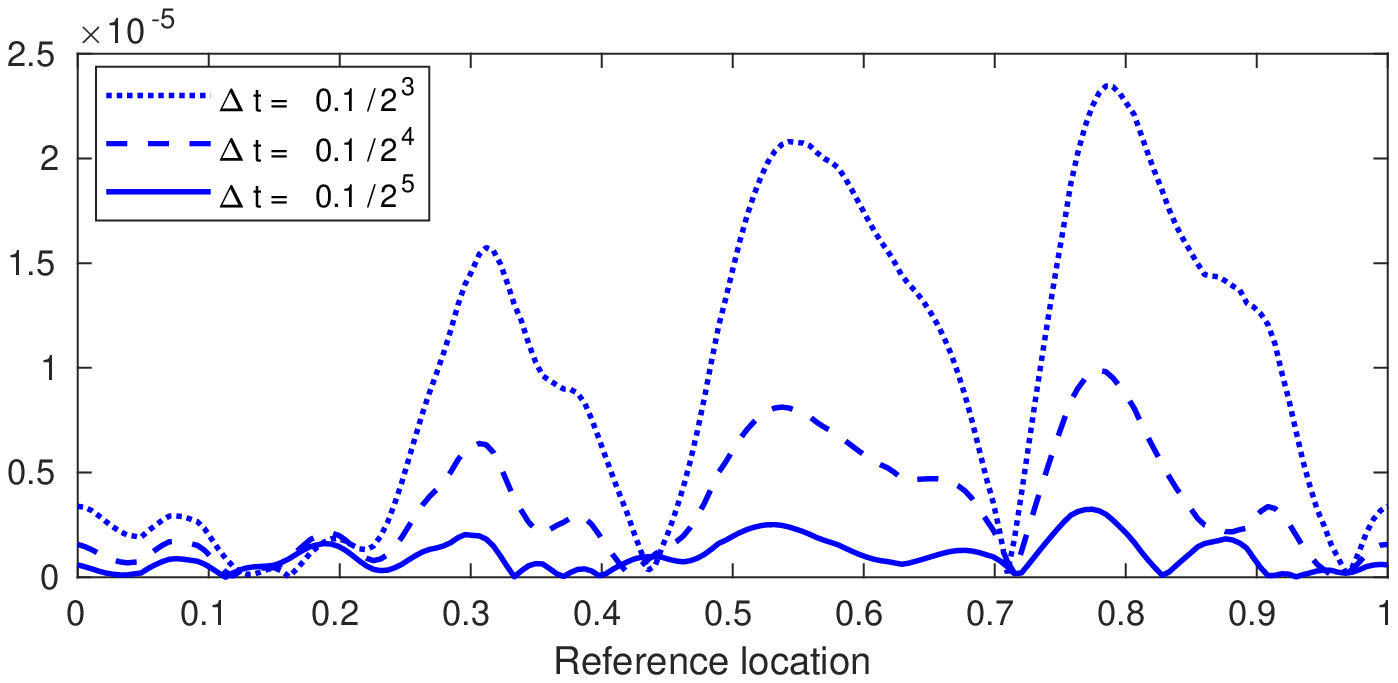} 
		\end{center}
		
		\caption{Plot of the difference between the interface location at $t=0.1$ computed using $\Delta t = 0.1/2^6$ and the time steps indicated in the plot with $\mu=10$ and $\kappa=1$.
		The initial interface configuration is that described in Example 3.
		The quantity plotted is the sum of the absolute value of the difference at each point midpoint on $\Gamma_\hs$.}
		\label{fig: heart interface difference}
	\end{figure}

\subsection{Example 4: Stretched circle}

The fourth example is chosen to emphasize the effects of a nonuniform tension around the perimeter of $\Gamma_\hs^0$.
The initial configuration is a circle of radius $\frac{1}{4}$ centered at $\left(\frac{1}{2},\frac{1}{2}\right)$ with its parameterization given by
\begin{equation}
\bX^0(s) = \frac{1}{4}\pmat{\cos(2\pi s) + 2\\ \sin(2\pi s) + 2}, \quad s\in [0,1].
\label{circle}
\end{equation}
Previously we chose equally-spaced points from the interval $[0,L]$ when the tension of each segment was arbitrary.
Since $\Gamma_\hs^0$ is a circle,  we must sample at nonuniform $s_j$ to prescribe a nonuniform tension on the edges of the polygon $\Gamma_\hs^0$.
The parameter values $\tilde s_j$ used to compute $\bX_j^0 = \bX^0(\tilde s_j)$ are the $m+1$ evenly spaced points $s_j\in [0,1]$ mapped by the cubic function
$$\tilde s_j = \frac{1}{5}\left(16s_j^3 - 24s_j^2 + 13s_j\right).$$

The result of this simulation, computed using the semi-implicit method, will be a leftward moving circle as a force tangent to the interface is applied to the fluid.
Equilibrium is obtained when the points on the circle are equally spaced and the total force applied to the fluid is zero.
Figure \ref{fig: stretched circle time steps} shows the position of the points on the interface at three time steps.
These plots highlight the leftward motion and the even distribution of the points on $\Gamma_\hs^n$ near steady state, at $t=1$.
In this example, the semi-implicit method is unconditionally stable.

\begin{figure}[]
	\centering
	\makebox[\textwidth][c]{\includegraphics[width=1.3\textwidth]{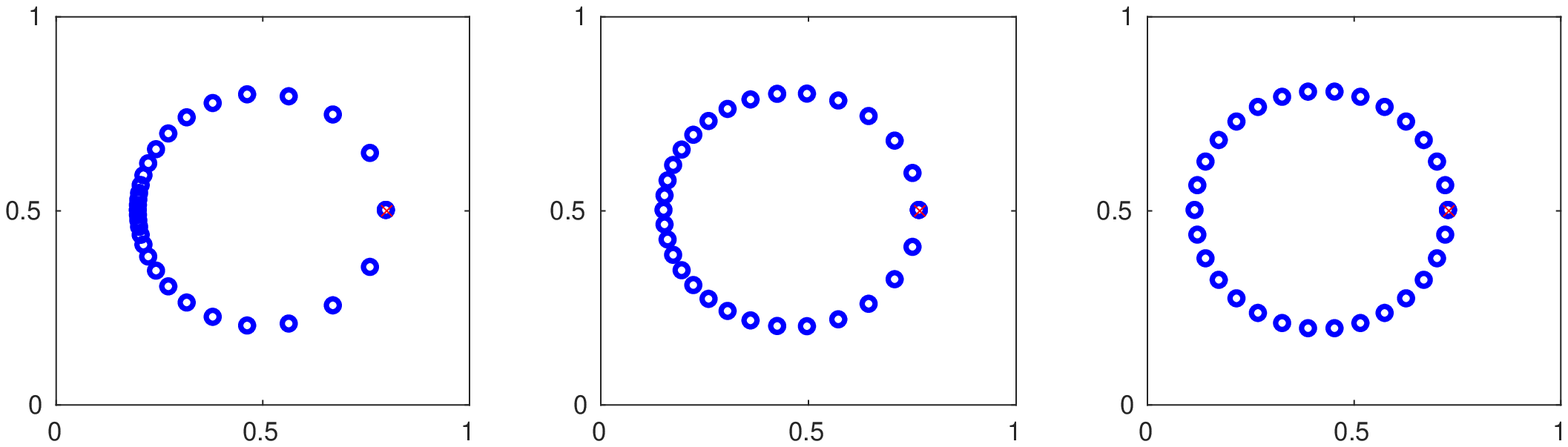}}
	\caption{Position of the interface at $t=0$, $t=0.05$, and $t=1$.
		We let $\mu = 1$, $\kappa = 10$, $\Delta t = 0.01$, and $N=32$ and use the initial configuration from Example 4.
		An encircled red $\times$ is used to denote $\bX(s_0,t)$ for reference.
		We display every 6th point on the interface to reduce clutter.
		\label{fig: stretched circle time steps}
	}
\end{figure}

\section{Conclusions}
\label{Sec: Conclusions}

In this work we presented a new finite element method for solving unsteady
Stokes equations with an immersed membrane that moves with the velocity of the fluid, not known a priori.
We successfully combined the classical immersed boundary method with Nitsche's formulation and CutFEM to solve this problem in two dimensions.
The proposed method maintains the use of the Dirac delta function to pass the force applied by the immersed structure in the Lagrangian frame to the fluid in the Eulerian frame.
more accurately incorporating the force applied by the interface on the fluid with conditional energy stability.
We developed a semi-implicit discretization and added the necessary consistent penalty terms to maintain energy stability.
The stability of our method is proved and verified in each example of our numerical results.
This semi-implicit method was tested alongside the explicit CutFEM method, which is the algorithm directly analogous to the original finite element immersed boundary method \cite{Boffi03}.
Using CutFEM we improved the error in computing the velocity and pressure of the fluid near the interface;
however, we continued the use of a polygonal approximation to the interface for its simplicity in computing the location of $\Gamma^{n+1}_\hs$.

The numerical results demonstrate that if the polygonal approximation to the interface is refined as the mesh is refined, we obtain optimal spatial convergence in Example 1, as shown in Table \ref{convergence table}.
We also observe in Examples 2-4 the theoretical unconditional energy stability proved in this work.
The conservation of area in each subdomain is desired and obtained for sufficiently small values of $h$ and $\Delta t$.
The trends observed in Table \ref{ellipse area table} and Table \ref{heart area table} are seen for sufficiently small values of $\Delta t$.
For larger values of $\Delta t$ we see an improvement in area conservation as $\Delta t$ is refined, however we may not observe such trends as $h$ is reduced.




\bibliographystyle{siam}
\bibliography{bib}

\end{document}